\documentclass[a4paper]{article}
  
\usepackage[inline]{enumitem}
  
  \usepackage[utf8]{inputenc}
\usepackage[T1]{fontenc}
\usepackage{amsmath,amsthm,amssymb,amsfonts}
\usepackage{xspace}
\usepackage{hyperref}
\usepackage{cleveref,soul}
\usepackage{verbatim}
\usepackage{tikz}
\usepackage{setspace}
\usetikzlibrary{decorations}

\usepackage[procnumbered,linesnumbered,ruled,vlined]{algorithm2e}

\usetikzlibrary{positioning, fadings, backgrounds}

\usepackage[textsize=footnotesize,color=green!40]{todonotes}
\usepackage[hmargin=2.5cm,vmargin=3cm]{geometry}

\usetikzlibrary{decorations.pathmorphing}
\tikzset{snake it/.style={decorate, decoration=snake}}

\newcommand{\IPC}{\textsc{Isometric Path Cover}\xspace}
\newcommand{\dist}[2]{\mathsf{d}\left(#1,#2\right)}

\newcommand{\anticp}[2]{A_{#1}\left(#2\right)}

\newcommand{\ipcor}[2]{ipco\left(\overrightarrow{#2_{#1}}\right)}

\newcommand{\ipac}[1]{ipacc\left(#1\right)}
\newcommand{\ipco}[1]{ipco\left(#1\right)}

\newcommand{\ie}{\textit{i.e.\xspace}}

\newcommand{\coverP}[2]{S_{#1}\left(#2\right)}

	\newenvironment{subproof}[1][\proofname]{%
  \begin{proof}[#1]%
}{%
  \end{proof}%
}
\newcommand{\Pnote}[2]{P\left(#1,#2\right)}

\newcommand{\pathseta}[3]{\mathcal{P}_{\searrow}^{#1}\left(#2,#3\right)}
\newcommand{\pathseteq}[3]{\mathcal{P}_{\rightarrow}^{#1}\left(#2,#3\right)}
\newcommand{\pathsetd}[3]{\mathcal{P}_{\nearrow}^{#1}\left(#2,#3\right)}

\definecolor{dartmouthgreen}{rgb}{0.05, 0.5, 0.06}

\newtheorem{theorem}{Theorem}

\newtheorem{proposition}[theorem]{Proposition}
\newtheorem{lemma}[theorem]{Lemma}
\newtheorem{observation}[theorem]{Observation}

\newtheorem{definition}[theorem]{Definition}

\newtheorem{claim}{Claim}[theorem]

    \title{Isometric path complexity of graphs\footnote{A preliminary version of this paper appeared as~\cite{shortversion} in the proceedings of MFCS 2023.}}

\author{Dibyayan Chakraborty\footnote{School of Computing, University of Leeds, United Kingdom. \texttt{Corresponding Author. Email:D.chakraborty@leeds.ac.uk}} \and Jérémie Chalopin\footnote{Laboratoire d'Informatique et Systèmes, Aix-Marseille Université and CNRS, Faculté des Sciences de Luminy, F-13288 Marseille, Cedex 9, France. This author was financed by the ANR projects DISTANCIA (ANR-17-CE40-0015) and DUCAT (ANR-20-CE48-0006).}
\and Florent Foucaud\footnote{Université Clermont Auvergne, CNRS, Clermont Auvergne INP, Mines Saint-Etienne, LIMOS, 63000 Clermont-Ferrand, France. This author was partially financed by the ANR project GRALMECO (ANR-21-CE48-0004), the French government IDEX-ISITE initiative 16-IDEX-0001 (CAP 20-25), and the International Research Center ``Innovation Transportation and Production Systems'' of the I-SITE CAP 20-25.} \and Yann Vax\`{e}s\footnote{Laboratoire d'Informatique et Systèmes, Aix-Marseille Université and CNRS, Faculté des Sciences de Luminy, F-13288 Marseille, Cedex 9, France. This author was financed by the ANR project DISTANCIA (ANR-17-CE40-0015)} }

\begin{document}

\maketitle

\newcommand{\correction}[1]{\textcolor{red}{#1}}
\newcommand{\ff}[1]{\textcolor{blue}{#1}}

\begin{abstract}

A set $S$ of isometric paths of a graph $G$ is ``$v$-rooted'', where $v$ is a vertex of $G$, if $v$ is one of the endpoints of all the isometric paths in $S$. The \emph{isometric path complexity} of a graph $G$, denoted by $\ipco{G}$, is the minimum integer $k$ such that there exists a vertex $v\in V(G)$ satisfying the following property: the vertices of any {single} isometric path $P$ of $G$ can be covered by $k$ many $v$-rooted isometric paths.

First, we provide an $O(n^2 m)$-time algorithm to compute the isometric path complexity of a graph with $n$ vertices and $m$ edges. Then we show that the isometric path complexity remains bounded for graphs in three seemingly unrelated graph classes, namely, \emph{hyperbolic graphs}, \emph{(theta, prism, pyramid)-free graphs}, and \emph{outerstring graphs}. 
There is a direct algorithmic consequence of having small isometric path complexity. Specifically, we show that if the isometric path complexity of a graph $G$ is bounded by a constant, then there exists a polynomial-time constant-factor approximation algorithm for \IPC, whose objective is to cover all vertices of a graph with a minimum number of isometric paths. This applies to all the above graph classes. 

\medskip\noindent\textbf{Keywords: Shortest paths, Isometric path complexity, Hyperbolic graphs, Truemper Configurations, Outerstring graphs, Isometric Path Cover}
\end{abstract}

\section{Introduction}

Finding properties of graph classes that can be exploited to develop
efficient algorithms is a popular direction of research in graph
theory. It has motivated researchers to define varieties of graph
parameters, study their combinatorial properties, and develop
efficient algorithms based on them.
Identifying graph classes where these parameters are bounded is also
fundamental to establish their relevance. Some of them
are bounded for seemingly unrelated classes,
e.g. \emph{mim-width}~\cite{BELMONTE201354,JAFFKE2020153} or
\emph{twin-width}~\cite{twinwidth}, showing unexpected structural
similarities between these classes.
In this paper, we study the \emph{isometric path complexity} that is
initially motivated by an algorithmic application. We show that it
can be computed in polynomial time and that it is bounded on three
 graph classes studied in different research areas in graph theory.


\smallskip

\noindent{\textbf{The isometric path cover problem.}} {Recently, Chakraborty et al.~\cite{ChakrabortyD0FG22} introduced a parameter called \emph{isometric path antichain cover number} of a graph $G$, denoted as $\ipac{G}$ (see Definition~\ref{D:acWidth}) in the context of an algorithmic problem known as the \IPC.}

A path is \emph{isometric} if it is a shortest path between its endpoints\footnote{{Sometimes isometric paths are also referred to as geodesic paths in the literature.}}.  An \emph{isometric path cover} of a graph $G$ is a set of isometric paths such that each vertex of~$G$ belongs to at least one of the paths. The \emph{isometric path number} of $G$ is the smallest size of an isometric path cover of $G$. Given a graph $G$ and an integer $k$, the objective of \IPC is to decide if there exists an isometric path cover of cardinality at most $k$. 

\IPC has been {introduced} and studied in the context of pursuit-evasion games~\cite{cop-decs,AF84}. However, until recently the algorithmic aspects of \IPC remained unexplored. After proving that \IPC remains NP-hard on \emph{chordal graphs} (graphs without any induced cycle of length at least 4), Chakraborty et al.~\cite{ChakrabortyD0FG22} provided constant-factor approximation algorithms for many graph classes, including \emph{interval graphs}, chordal graphs, and more generally, graphs with bounded \emph{treelength}. {The authors proved the bound on the approximation ratio by showing that the isometric path antichain cover number of the above graph classes is bounded by a constant.} Specifically, they proved that $(i)$ when $\ipac{G}$ is bounded by a constant, \IPC admits a constant-factor approximation algorithm on $G$; and $(ii)$ the isometric path antichain cover number of graphs with bounded \emph{treelength} is bounded. 

\smallskip
{\noindent\textbf{Our objectives.}} The objective of this paper is three-fold: \textbf{(A)} provide a more intuitive definition of isometric path antichain cover number; \textbf{(B)} provide a polynomial-time algorithm to compute $\ipac{G}$; and \textbf{(C)} prove that it remains bounded for seemingly unrelated graph classes.  Along the way, we also extend the horizon of approximability of \IPC. To achieve \textbf{(A)} we introduce the following new metric graph parameter, that we will show to be always equal to the isometric path antichain cover number, and whose definition is simpler. 

\begin{definition}\label{D:ipco}
    Given a graph $G$ and a vertex $v$ of 
    $G$, a set $S$ of isometric paths of $G$ is \emph{$v$-rooted} if $v$ is one of the endpoints of all the isometric paths in $S$. The \emph{isometric path complexity} of a graph $G$, denoted by $\ipco{G}$, is the minimum integer $k$ such that there exists a vertex $v\in V(G)$ satisfying the following property: the vertices of any isometric path $P$ of $G$ can be covered by $k$ many $v$-rooted isometric paths.
\end{definition}

A consequence of Dilworth's theorem~\cite{D50} is that for any graph $G$, $\ipac{G} = \ipco{G}$ (see Lemma~\ref{lem:ipac-ipco}). We will give a polynomial-time algorithm to compute $\ipco{G}$, and therefore $\ipac{G}$ for an arbitrary undirected graph $G$. This achieves \textbf{(B)}.

Finally, to achieve \textbf{(C)}, we consider the following three seemingly unrelated graph classes, namely, \emph{$\delta$-hyperbolic graphs}, \emph{(theta, prism, pyramid)-free graphs} and \emph{outerstring} graphs, and show that their isometric path complexity is bounded by a constant.

\begin{figure}[t]
\centering
\scalebox{0.6}{\begin{tikzpicture}[node distance=7mm]

\tikzstyle{mybox}=[fill=white,line width=0.5mm,rectangle, minimum height=.8cm,fill=white!70,rounded corners=1mm,draw];
\tikzstyle{myedge}=[line width=0.5mm]
\newcommand{\tworows}[2]{\begin{tabular}{c}{#1}\\{#2}\end{tabular}}

     \node[mybox] (ipacc)  {\begin{tabular}{c}
          Bounded isometric path  \\
          complexity
     \end{tabular}};
     
     \node[mybox] (hyp) [below left=of ipacc,xshift=-1cm] {bounded hyperbolicity \textbf{*}} edge[myedge] (ipacc);

     \node[mybox] (theta) [below right=of ipacc,xshift=1cm] {\begin{tabular}{c}
          ($t$-theta, $t$-prism, $t$-pyramid)-free \textbf{*}
     \end{tabular}}  edge[myedge] (ipacc);
     
    \node[mybox] (outer) [below left=of theta,xshift=-0.5cm, yshift=-0.15cm] {\begin{tabular}{c}
          Outerstring \textbf{*} \\
          
     \end{tabular}} edge[myedge] (theta) ;
     
     \node[mybox, fill=gray!20] (circle) [below left=of outer,xshift=-0.2cm, yshift=-0.35cm] {\begin{tabular}{c}
          circle \textbf{*}
    \end{tabular}}  edge[myedge] (outer);

    \node[mybox] (truemper) [below right=of theta,xshift=0.5cm] {\begin{tabular}{c}
          (theta,prism,pyramid)-  \\
          free \textbf{*}
    \end{tabular}} edge[myedge] (theta) ;

     \node[mybox] (univ) [below =of truemper] {\begin{tabular}{c}
          Universally signable \textbf{*}
    \end{tabular}} edge[myedge] (truemper);
    
     \node[mybox] (treelength) [below=of hyp,] {bounded tree-length} edge[myedge] (hyp);
     
     \node[mybox] (chordality) [below=of treelength, yshift=-0.4cm] {bounded chordality} edge[myedge] (treelength);
  
     \node[mybox] (bdiameter) [below=of treelength, xshift=-4cm, yshift=-0.4cm] {bounded diameter} edge[myedge] (treelength);

    \node[mybox] (chordal) [below=of chordality, yshift=-0.5cm] {chordal} edge[myedge] (chordality); \draw[myedge] (chordal.east) -| (outer.south) ; \draw[myedge] (chordal.east) -| (univ.south) ;
     
    \node[mybox, fill=gray!20] (atfree) [below left=of chordality,xshift=-0.5cm, yshift=-0.5cm] {AT-free} edge[myedge] (chordality);
     
    \node[mybox, fill=gray!20] (interval) [below =of chordality,yshift=-2cm] {Interval} edge[myedge] (chordal) edge[myedge] (atfree);

     \node[mybox, fill=gray!20] (arc) [below =of theta, yshift=-1.75cm, xshift=0.11cm] {\begin{tabular}{c}
          circular arc \textbf{*}
    \end{tabular} } edge[myedge] (outer) edge[myedge] (truemper);
    
    \draw[myedge] (interval.east) -- ++ (11.8,0) -- ++(0,1.25) -- ++ (0.5,0) -- ++ (0,0.5) -- ++ (-0.5,0) -- (arc.south);
    
    \node[mybox, fill=gray!20] (perm) [below =of interval] {Permutation}; \draw[myedge] (perm.west) -| (atfree.south); \draw[myedge] (perm.east) -- ++ (2.9,0) -- ++ (0, 1.3) -- ++ (0.5,0) -- ++ (0,0.5) -- ++ (-0.5,0) -- ++ (0,1) -- ++ (0.5,0) -- ++ (0,0.5) -- ++ (-0.5,0) -- (circle.south);
    
  \end{tikzpicture}}

\caption{Inclusion diagram for graph classes.
If a class $A$ has an upward path to class $B$, then $A$ is included in $B$. Constant bounds for the isometric path complexity
on graph classes marked with \textbf{*} are contributions of this paper.}
\label{fig:diagram}
\end{figure}

 \smallskip
 \noindent
\textbf{$\delta$-hyperbolic graphs:} A graph $G$ is said to be \emph{$\delta$-hyperbolic}~\cite{gromov1987} if for any four vertices $u,v,x,y$, the two larger of the three distance sums $\dist{u}{v}+\dist{x}{y}$, $\dist{u}{x}+\dist{v}{y}$ and $\dist{u}{y}+\dist{v}{x}$ differ by at most $2\delta$. A graph class $\mathcal{G}$ is \emph{hyperbolic} if there exists a constant $\delta$ such that every graph $G\in \mathcal{G}$ is $\delta$-hyperbolic. This parameter comes from geometric group theory and  was first introduced by Gromov~\cite{gromov1987} in order to study groups via their \emph{Cayley graphs}. The hyperbolicity of a tree is $0$, and in general, the hyperbolicity measures how much the distance function of a graph deviates from a tree metric.   Many structurally defined graph classes like chordal graphs, \emph{cocomparability} graphs~\cite{corneil2013ldfs}, \emph{asteroidal-triple free} graphs~\cite{corneil1997asteroidal}, graphs with bounded \emph{chordality} or \emph{treelength} are hyperbolic~\cite{chepoi2008diameters,kosowski2015k}. Moreover, hyperbolicity has been found to capture important properties of several large practical graphs such as the Internet graph~\cite{shavitt2004curvature} or database relation graphs~\cite{walter2002interactive}. Due to its importance in discrete mathematics, algorithms, \emph{metric graph theory}, researchers have studied various algorithmic aspects of hyperbolic graphs~\cite{chepoi2008diameters,coudert2021enumeration,chepoi2017core,das2018effect}. Note that graphs with diameter~2 are hyperbolic, which may contain any graph as an induced subgraph.

 \smallskip
\noindent
\textbf{(theta, prism, pyramid)-free graphs:} A \emph{theta} is a graph made of three vertex-disjoint induced paths $P_1 = a\ldots b$, $P_2 = a\ldots b$, $P_3 = a\ldots b$ of lengths at least~2, and such that no edges exist between the paths except the three edges incident to $a$ and the three edges incident to $b$. 
A \emph{pyramid} is a graph made of three induced paths $P_1 = a\ldots b_1$, $P_2 = a\ldots b_2$, $P_3 = a\ldots b_3$, two of which have lengths at least $2$, vertex-disjoint except at $a$, and such that $b_1 b_2 b_3$ is a triangle and no edges exist between the paths except those of the triangle and the three edges incident to $a$. A \emph{prism} is a graph made of three vertex-disjoint induced paths $P_1 = a_1\ldots b_1$, $P_2 = a_2\ldots b_2$, $P_3 = a_3\ldots b_3$ of lengths at least $1$, such that $a_1 a_2 a_3$ and $b_1 b_2b_3$ are triangles and no edges exist between the paths except those of the two triangles. A graph $G$ is \emph{(theta, pyramid, prism)}-free if $G$ does not contain any induced subgraph isomorphic to a theta, pyramid or prism. A graph is a \emph{$3$-path configuration} if it is a theta, pyramid or prism. The study of $3$-path configurations dates back to the works of Watkins and Meisner~\cite{watkins1967cycles} in 1967 and plays ``special roles'' in the proof of the celebrated \emph{Strong Perfect Graph Theorem}~\cite{chudnovsky2006strong}. Important graph classes like chordal graphs, \emph{circular arc} graphs, \emph{universally-signable} graphs~\cite{conforti1997universally} exclude all $3$-path configurations. Popular graph classes like \emph{perfect} graphs, \emph{even hole}-free graphs exclude some {(but not all)} of the $3$-path configurations. Note that, (theta, prism, pyramid)-free graphs are not hyperbolic. To see this, consider a cycle $C$ of order $n$. Clearly, $C$ excludes all $3$-path configurations and has hyperbolicity $\Omega(n)$.

 \smallskip
 \noindent
\textbf{Outerstring graphs:} A set $S$ of simple curves on the plane is \emph{grounded} if there exists a horizontal line containing one endpoint of each of the curves in $S$. A graph $G$ is an \emph{outerstring} graph if there is a collection $C$ of grounded simple curves and a bijection between $V(G)$ and $C$ such that two curves in $S$ intersect if and only if the corresponding vertices are adjacent in $G$. 
The term ``outerstring graph'' was first used in the early 90's~\cite{kratochvil1991string} in the context of studying intersection graphs of simple curves on the plane. Many well-known graph classes like chordal graphs, \emph{circular arc} graphs~\cite{francis2014forbidden}, \emph{circle} graphs (intersection graphs of chords of a circle~\cite{davies2021circle}), and cocomparability graphs~\cite{corneil2013ldfs} are also outerstring graphs and thus, motivated researchers from the \emph{geometric graph theory} and \emph{computational geometry} communities to study algorithmic and structural aspects of outerstring graphs and its subclasses~\cite{biedl2018size,bose2022computing,cardinal2017intersection,keil2017algorithm,rok2019outerstring}. Note that in general, outerstring graphs may contain a prism, pyramid or theta as an induced subgraph. Moreover, cycles of arbitrary order are outerstring graphs, implying that outerstring graphs are not hyperbolic.

\smallskip

 It is clear from the above discussion that the classes of hyperbolic graphs, (theta, prism, pyramid)-free graphs, and outerstring graphs are {pairwise} incomparable (with respect to the containment relationship). We show that the isometric path complexities of all the above graph classes are small.

\subsection*{Our contributions}

The main technical contribution of this paper are as follows. First we prove that the isometric path complexity can be computed in polynomial time.

\begin{theorem}
    \label{thm:ipcoInP}
    Given a graph $G$ with $n$ vertices and $m$ edges, it is possible to compute $\ipco{G}$ in $O(n^2m)$ time.
\end{theorem}


Then we show that the isometric path complexity remains bounded on hyperbolic graphs, (theta, pyramid, prism)-free graphs, and outerstring graphs. Specifically, we prove the following theorem.

\begin{theorem}\label{thm:main}
Let $G$ be a graph.
\vspace{-5pt}
    \begin{enumerate}[label=(\alph*)]
        \item \label{thm:hyperbolicity} If the hyperbolicity of $G$ is at most $\delta$, then $\ipco{G} \leq  4\delta+3$.
        \item\label{thm:truemper} If $G$ is a (theta, pyramid, prism)-free graph, then $\ipco{G} \leq 71$.
        \item\label{thm:outer} If $G$ is an outerstring graph, then $\ipco{G} \leq 95$.
    \end{enumerate}
\end{theorem}

To the best of our knowledge, the isometric path complexity being bounded (by constant(s)) is the only known non-trivial property shared by any two or all three of these graph classes. Theorem~\ref{thm:main} shows that isometric path complexity (equivalently isometric path antichain cover number), as recently introduced graph parameters, are general enough to unite these three graph classes by their metric properties. We hope that this definition will be useful for the field of metric graph theory, for example by enabling us to study (theta,prism,pyramid)-free graphs and outerstring graphs from the perspective of metric graph theory. 

We provide a unified proof for Theorem~\ref{thm:main}\ref{thm:truemper} and~\ref{thm:main}\ref{thm:outer} by proving that the isometric path complexity of \emph{($t$-theta, $t$-pyramid, $t$-prism)}-free graphs~\cite{trotignonprivate} (see Section~\ref{sec:main} for a definition) is bounded by a linear function of $t$. Due to the above theorems, we also have as corollaries that there is a {polynomial-time} approximation algorithm for \IPC with approximation ratio
\begin{enumerate*}[label=(\alph*)]
    \item $4\delta+3$ on $\delta$-hyperbolic graphs,
    \item $73$ on (theta, prism, pyramid)-free graphs,
    \item $95$ on outerstring graphs, and
    \item $8t+63$ on ($t$-theta, $t$-pyramid, $t$-prism)-free graphs.
\end{enumerate*}

To contrast with Theorem~\ref{thm:main}, we construct highly structured graphs with small \emph{tree-width} and large isometric path complexity. A \emph{wheel} consists of an induced cycle $C$ of order at least $4$ and a vertex $w \notin V(C)$ adjacent to at least three vertices of $C$. The 3-path configurations introduced earlier and the wheel together are called \emph{Truemper configurations}~\cite{vuskovic2013world} and they are important objects of study in structural and algorithmic graph theory~\cite{aboulker2015wheel,diot2020theta}.

\begin{theorem}\label{thm:lower}
For every $k\geq 1$, 
\vspace{-7.5pt}
\begin{enumerate}[label=(\alph*)]
    \item\label{it:a} there exists a (pyramid, prism, wheel)-free graph $G$ with tree-width $2$, hyperbolicity at least $\lceil\frac{k}{2}\rceil-1$ and $\ipco{G}\geq k$;
    \item\label{it:b} there exists a (theta, prism, wheel)-free planar graph $G$ with tree-width at most $3$, hyperbolicity at least $\lceil\frac{k}{2}\rceil-1$ and $\ipco{G}\geq k$;
    \item\label{it:c} there exists a (theta, pyramid, wheel)-free planar graph $G$ with hyperbolicity at least $\lceil\frac{k}{2}\rceil-1$ and $\ipco{G}\geq k$;
    \item \label{it:d} { there exists a (prism, pyramid, wheel)-free planar bipartite graph $G$ such that $|V(G)|$ is $O(k^2)$, $G$ has an isometric path cover of size $3k+1$ and any $v$-rooted isometric path cover of $G$ has cardinality at least $k^2$ for any $v\in V(G)$. }
\end{enumerate}
\end{theorem}

 {  
Theorem~\ref{thm:lower}\ref{it:d} proves that the approximation algorithm for \textsc{Isometric Path Cover} proposed by Chakraborty et al.~\cite{ChakrabortyD0FG22} cannot provide a $o(\sqrt{n})$ approximation ratio (even if the inputs are restricted to planar bipartite graphs of order $n$). Note that previous known lower bound (stated in~\cite{ChakrabortyD0FG22}) was $o(\sqrt{\log n})$.}

\medskip

%

 \noindent \textbf{Organisation.} In Section~\ref{sec:prelim}, we recall some definitions and some results. In Section~\ref{sec:ipcoInP}, we present an algorithm to compute the isometric path complexity of a graph and prove Theorem~\ref{thm:ipcoInP}. In Section~\ref{sec:main}, we prove Theorem~\ref{thm:main}. In Section~\ref{sec:lower}, we prove Theorem~\ref{thm:lower}. We conclude in Section~\ref{sec:conclu}.

\section{Definitions and preliminary observations}\label{sec:prelim}

In this section, we recall some definitions and some related observations. A sequence of distinct vertices forms a \emph{path} $P$ if any two consecutive vertices are adjacent. 
Whenever we fix a path $P$ of $G$, we shall refer to the subgraph formed by the edges between the consecutive vertices of $P$. The \emph{length} of a path $P$, denoted by $|P|$, is the number of its vertices minus  one. A path is \emph{induced} if there are no graph edges joining non-consecutive vertices. A path is \emph{isometric} if it is a shortest path between its endpoints.  For two vertices $u,v$ of a graph $G$, $\dist{u}{v}$ denotes the length of an isometric path between $u$ and $v$. 

\textcolor{black}{For a path $P$ and a vertex $u\in V(P)$, we will often use the phrase ``$(u,w)$-subpath'' of $P$ to denote a subpath that satisfies certain properties (which will be clear from the context), and has $u,w$ as endpoints, where $w$ is the endpoint distinct from $u$. We note that $w$ and the subpath are getting defined simultaneously, but for the sake of readability, we use the above notation.}

{In a directed graph, a \emph{directed path} is a path in which all arcs are oriented in the same direction.} For a path $P$ of a graph $G$ between two vertices $u$ and $v$, the vertices $V(P)\setminus \{u,v\}$ are \emph{internal vertices} of $P$. A path between two vertices $u$ and $v$ is called a $(u,v)$-path.  Similarly, we have the notions of \emph{isometric $(u,v)$-path} and \emph{induced $(u,v)$-path}. The interval $I(u,v)$ between two vertices $u$ and $v$ consists of all vertices that belong to an isometric $(u,v)$-path. For a vertex $r$ of $G$ and a set $S$ of vertices of $G$, the \emph{distance of $S$ from $r$}, denoted as $\dist{r}{S}$, is the minimum of the distance between any vertex of $S$ and $r$. For a subgraph $H$ of $G$, the \emph{distance of $H$ w.r.t. $r$} is $\dist{r}{V(H)}$. Formally, we have $\dist{r}{S}=\min\{\dist{r}{v}\colon v\in  S\}$ and $\dist{r}{H}=\dist{r}{V(H)}$.  

 For a graph $G$ and a vertex $r \in V(G)$, consider the following operations on $G$. First, remove all edges $xy$ from $G$ such that $\dist{r}{x}=\dist{r}{y}$. Let $G'_r$ be the resulting graph. Then, for each edge $e=xy\in E(G'_r)$ with $\dist{r}{x} = \dist{r}{y} - 1$, orient $e$ from $y$ to $x$. Let $\overrightarrow{G_r}$ be the directed acyclic graph formed after applying the above operation on $G'$. Note that this digraph can easily be computed in linear time using a Breadth-First Search (BFS) traversal with starting vertex $r$. 
 
{  The known approximation algorithm for \IPC from~\cite{ChakrabortyD0FG22} can now be stated as follows: $(i)$ For each vertex $r\in V(G)$, compute $\overrightarrow{G_r}$ and find a minimum path cover $\mathcal{C}_r$ of $\overrightarrow{G_r}$, and then $(ii)$ report a $\mathcal{C}_r$ with minimum cardinality. {The following definition is inspired by the terminology of posets (as the graph $\overrightarrow{G_r}$ can be seen as the Hasse diagram of a poset) and will be useful to analyze the above algorithm.}}

\begin{definition}\label{D:acSet}
\sloppy For a graph $G$ and a vertex $r\in V(G)$, two vertices $x,y\in V(G)$ are \emph{antichain vertices} w.r.t $r$, if there are no directed paths from $x$ to $y$ or from $y$ to $x$ in $\overrightarrow{G_r}$.
\end{definition}

For a graph $G$ and a vertex $r\in V(G)$, a set $X$ of vertices of $G$ is an \emph{antichain set} if any two vertices in $X$ are antichain vertices w.r.t $r$. 
 
\begin{definition}[{\cite{ChakrabortyD0FG22}}]\label{D:acWidth}
Let $r$ be a vertex of a graph $G$. For a subgraph $H$, $\anticp{r}{H}$ shall denote the maximum antichain set of $H$ in $\overrightarrow{G_r}$. The \emph{isometric path antichain cover number} of $\overrightarrow{G_r}$, denoted by $\ipac{\overrightarrow{G_r}}$, is defined as follows: \[\ipac{\overrightarrow{G_r}}=\max\left\{|\anticp{r}{P}|\colon~P~\text{is an isometric path in }G\right\}.\]
The \emph{isometric path antichain cover number} of graph $G$, denoted as $\ipac{G}$, is defined as the minimum over all possible antichain covers of its associated directed acyclic graphs: \[\ipac{G}=\min \left\{\ipac{\overrightarrow{G_r}}\colon r\in V(G)\right\}.\]
\end{definition}

For technical purposes, we also introduce the following definition. For a graph $G$ and a vertex $r$ of $G$, let $\ipcor{r}{G}$ denote the minimum integer $k$ such that any isometric path $P$ of $G$ can be covered by $k$ $r$-rooted isometric paths (The notation reflects that it is a dual notion of $\ipac{\overrightarrow{G_r}}$). Using Dilworth's Theorem~\cite{D50} we prove the following important lemma.

\begin{lemma}\label{lem:ipac-ipco}
    For any graph $G$ and vertex $r$, $\ipcor{r}{G} = \ipac{\overrightarrow{G_r}}$. Therefore,  $\ipco{G}=\ipac{G}$.
\end{lemma}

\begin{proof}
Let $r$ be a vertex of $G$ such that any isometric path of $G$ can be
covered by $\ipcor{r}{G}$ $r$-rooted isometric paths. Let $P$ be an
arbitrary isometric path of $G.$ Since two vertices of an antichain of
$\overrightarrow{G_r}$ cannot be covered by a single $r$-rooted path and
$P$ is covered by $\ipcor{r}{G}$ many $r$-rooted paths, we deduce $|\anticp{r}{P}|\le
\ipcor{r}{G}$. This is true for any isometric path $P$ of $G$. Hence,
$ \ipac{\overrightarrow{G_r}} \leq \ipcor{r}{G}$. Conversely, consider a vertex $r\in V(G).$ By definition of $\ipcor{r}{G}$,
there is an isometric path $P$ that cannot be covered by $(\ipcor{r}{G}-1)$
$r$-rooted isometric paths. By Dilworth's theorem~\cite{D50}, $P$ contains an
antichain of $\overrightarrow{G_r}$ of size $\ipcor{r}{G}.$ Hence
$|A_r(P)|\geq \ipcor{r}{G}$ and $\ipac{\overrightarrow{G_r}}\geq \ipcor{r}{G}$.
The second part of the lemma follows immediately.
\end{proof}

We also recall the following theorem and proposition from~\cite{ChakrabortyD0FG22}. 

\begin{theorem}[\cite{ChakrabortyD0FG22}]\label{thm:ipac-approx}
For a graph $G$, if $\ipac{G} \leq c$, then \IPC admits a {polynomial-time} $c$-approximation algorithm on $G$.
\end{theorem}


\begin{proposition}[\cite{ChakrabortyD0FG22}]\label{prp:antichain-length}
Let $G$ be a graph and $r$, an arbitrary vertex of $G$. Consider the directed acyclic graph $\overrightarrow{G_r}$, and let $P$ be an isometric path between two vertices $x$ and $y$ in $G$. Then $|P|\geq |\dist{r}{x}-\dist{r}{y}| + |\anticp{r}{P}| - 1$. 
\end{proposition}

\begin{proof}
Orient the edges of $P$ from $y$ to $x$ in $G$. First, observe that $\overrightarrow{P}$ must contain a set $E_1$ of oriented edges such that $|E_1|=|\dist{r}{y}-\dist{r}{x}|$ and for any $\overrightarrow{ab}\in E_1$, $\dist{r}{a}=\dist{r}{b}+1$. Let the vertices of the largest antichain set of $P$ in $\overrightarrow{G_r}$, \ie, $\anticp{r}{P}$, be ordered as $a_1,a_2,\ldots,a_t$ according to their occurrence while traversing $P$ from $y$ to $x$. For $i\in [2,t]$, let $P_i$ be the subpath of $P$ between $a_{i-1}$ and $a_i$. Observe that for any $i\in [2,t]$, since $a_i$ and $a_{i-1}$ are antichain vertices, there must exist an oriented edge $\overrightarrow{b_ic_i}\in E(\overrightarrow{P_i})$ such that either $\dist{r}{b_i} = \dist{r}{c_i}$ or $\dist{r}{b_i}=\dist{r}{c_i} - 1$. Let $E_2=\{b_ic_i\}_{i\in [2,t]}$. Observe that $E_1\cap E_2=\emptyset$ and therefore $|P|\geq |E_1| + |E_2| = |\dist{r}{y}-\dist{r}{x}| + |\anticp{r}{P}| - 1$.
\end{proof}

\section{Proof of Theorem \ref{thm:ipcoInP}} \label{sec:ipcoInP}

In this section we provide a polynomial-time algorithm to compute the
isometric path complexity of a graph. Let $G$ be a graph.  In the
following lemma, we provide a necessary and sufficient condition for
two vertices of an isometric path to be covered by the same isometric
$r$-rooted path in $\overrightarrow{G_r}$ for some vertex $r\in V(G)$.

\begin{lemma}\label{lemma1}
  Let $r$ be vertex of $G$. If $P=(u=v_0,\dots,v_k=v)$ is an isometric
  $(u,v)$-path with $\dist{r}{u}\le \dist{r}{v}$ then there exists an
  isometric $r$-rooted path containing $u, v$ in
  $\overrightarrow{G_r}(P)$ if and only if
  $\dist{v_{i+1}}{r} = \dist{v_i}{r}+1$ for all
  $i\in \{0,\dots,k-1\}.$
\end{lemma}


\begin{proof}
  If $\dist{v_{i+1}}{r}= \dist{v_i}{r}+1$ for every
  $i\in \{0,\dots,k-1\}$ then the path obtained by concatenating an
  isometric $(r,u)$-path and the path $P$ is an isometric $r$-rooted
  $(r,v)$-path containing $u, v$ in $\overrightarrow{G_r}(P)$.  Now
  suppose that there exists an isometric $r$-rooted path containing
  $u, v$ in $\overrightarrow{G_r}(P)$, \ie,
  $\dist{r}{v}-\dist{r}{u}=\dist{u}{v}.$ Then, along any path from $u$
  to $v$, we need to traverse at least $\dist{u}{v}$ edges increasing
  the distance to $r$. Since $P$ is an isometric $(u,v)$-path, it
  contains exactly $\dist{u}{v}$ edges. Hence,
  $\dist{r}{v_{i+1}}=\dist{r}{v_i}+1$ for every
  $i\in \{0,\dots,k-1\}$.
\end{proof}

\subsection{Notations and preliminary observations}

We now introduce some notations that will be used to describe the
algorithm and prove its correctness.  Consider three vertices $r,x,v$
of $G$ such that $x \neq v$. Let $\pathseta{r}{x}{v}$ denote the set
of all isometric $(x,v)$-paths $P$ containing a vertex $u$ that is
adjacent to $v$ and satisfies
$\dist{r}{u}=\dist{r}{v}-1$. Analogously, let $\pathseteq{r}{x}{v}$
denote the set of all isometric $(x,v)$-paths $P$ containing a vertex
$u$ that is adjacent to $v$ and satisfies $\dist{r}{u}=\dist{r}{v}$
and let $\pathsetd{r}{x}{v}$ denote the set of all isometric
$(x,v)$-paths $P$ containing a vertex $u$ that is adjacent to $v$ and
satisfies $\dist{r}{u}=\dist{r}{v}+1$. Observe that the set of
isometric $(x,v)$-paths is precisely
$\pathseta{r}{x}{v} \cup \pathseteq{r}{x}{v} \cup \pathsetd{r}{x}{v}$
and that some of these sets may be empty.

Given a path $P$, we denote by $|\coverP{r}{P}|$ the minimum size of a
set of isometric $r$-rooted paths covering the vertices of $P$.  We
denote by $\gamma^r_{\searrow}(x,v)$ and $\beta^r_{\searrow}(x,v)$
respectively the \textcolor{black}{maximum} of $|\coverP{r}{P}|$ and
$|\coverP{r}{P-\{v\}}|$ over all paths $P\in \pathseta{r}{x}{v}$. More
formally,
\begin{align*}
  \gamma^r_{\searrow}(x,v)&=\max\left\{ |\coverP{r}{P}| \colon P \in \pathseta{r}{x}{v} \right\}, \\
  \beta^r_{\searrow}(x,v)&=\max\left\{ |\coverP{r}{P - \{v\}}|\colon P\in \pathseta{r}{x}{v} \right\}.
\end{align*}
Note that if $\pathseta{r}{x}{v}$ is empty, we have
$\gamma^r_{\searrow}(x,v) = \beta^r_{\searrow}(x,v) = 0$.  We define
similarly $\gamma^r_{\nearrow}(x,v)$, $\beta^r_{\nearrow}(x,v)$, and
$\gamma^r_{\rightarrow}(x,v)$:
\vspace{-7.5pt}
\begin{align*}
  \gamma^r_{\nearrow}(x,v)&=\max\left\{ |\coverP{r}{P}|\colon P\in \pathsetd{r}{x}{v} \right\},\\ 
  \beta^r_{\nearrow}(x,v)&=\max\left\{ |\coverP{r}{P- \{v\}}|\colon P\in \pathsetd{r}{x}{v} \right\},\\
  \gamma^r_{\rightarrow}(x,v)&=\max\left\{ |\coverP{r}{P}|\colon P\in \pathseteq{r}{x}{v} \right\}. %
\end{align*}
Finally, let
$\gamma^r(x,v) = \max\left\{ \gamma^r_{\searrow}(x,v),
  \gamma^r_{\rightarrow}(x,v) , \gamma^r_{\nearrow}(x,v) \right\} $ be
the maximum of $|S_r(P)|$ over all isometric $(x,v)$-paths $P$.
In our algorithm, we will need also to consider the case where $v=x$
as an initial case. For practical reasons, we let
$\gamma^r(x,x) = \gamma^r_{\searrow}(x,x) =
\gamma^r_{\rightarrow}(x,x) = \gamma^r_{\nearrow}(x,x) = 1$ and
$\beta^r_{\searrow}(x,x) = \beta^r_{\nearrow}(x,x) =0$.
Based on the above notations and Lemma~\ref{lem:ipac-ipco}, we have the following observation.

\begin{observation}\label{obs:triv-poly}
  For any graph $G$ and any vertex $r$ of $G$, we have
  $\ipco{\overrightarrow{G_r}} = \ipac{\overrightarrow{G_r}} =
  \max_{x,v} \gamma^r(x,v)$ and
  $\ipco{G} = \ipac{G} = \min_{r} \max_{x,v} \gamma^r(x,v)$.
\end{observation}

Observation~\ref{obs:triv-poly} implies that to compute the isometric
path complexity of a graph it is enough to compute the parameter
$\gamma^r(x,v)$ for all $r,x,v\in V(G)$ in polynomial time. In the
next section, we focus on achieving this goal without computing
explicitly any of the sets $\pathseta{r}{x}{v}$, $\pathseteq{r}{x}{v}$
or $\pathsetd{r}{x}{v}$. (Note that the size of these sets could be
exponential in the number of vertices of the graph).

\subsection{An algorithm to compute \boldmath{$\gamma^r(x,v)$}}

Throughout this section, let $r$ and $x$ be two fixed vertices of $G$. We shall call $r$ as the ``root'' and $x$ as the ``source'' vertex. The objective of this section is to compute the parameter $\gamma^r(x,v)$ for all vertices $v\in V(G)$. 

In the sequel, since we always refer to a fixed root $r$ and source
$x$, we omit $r$ and $x$ and use the shorthand $\gamma(v)$ for
$\gamma^r(x,v).$ We do the same with the notations
$\gamma_{\nearrow}(v)$, $\gamma_{\rightarrow}(v)$,
$\gamma_{\searrow}(v)$, $\beta_{\nearrow}(v)$, and
$\beta_{\searrow}(v)$ that also refer to fixed vertices $r$ and
$x$
In the following lemmas, we shall provide explicit (recursive)
formulas to compute $\gamma_{\nearrow}(v)$, $\gamma_{\rightarrow}(v)$,
$\gamma_{\searrow}(v)$, $\beta_{\nearrow}(v)$, and
$\beta_{\searrow}(v)$. Using these formulas, we will show how to
compute $\gamma(v)$ for all $v\in V(G)$ in a total of
$O(|E(G)|)$-time.


\begin{observation}\label{lem:beta-arrow}
  If $r$ is the root vertex, $x$ the source vertex, and $v$ is
  distinct from $x$, then
  \begin{align*}
    \beta_{\searrow}(v) &= \max \{ \gamma(u) : u\in I(x,v) \cap N(v);\
                          \dist{r}{u}=\dist{r}{v}-1 \},\\
    \beta_{\nearrow}(v) &= \max \{ \gamma(u) : u\in I(x,v) \cap N(v);\
                          \dist{r}{u}=\dist{r}{v}+1 \}.
  \end{align*}
\end{observation}






\begin{lemma}\label{lem:gamma-rightarrow}
  If $r$ is the root vertex, $x$ the source vertex, and $v$ is
  distinct from $x$, then {\color{black} $\gamma_{\rightarrow}(v) = 0$  if $\pathseteq{r}{x}{v}=\emptyset$ and}
  $\gamma_{\rightarrow}(v) = \max \{ 1+\gamma(u) : u\in I(x,v) \cap
  N(v);\ \dist{r}{u}=\dist{r}{v} \}$
  {\color{black} otherwise}.
\end{lemma}

\begin{proof}
  Observe that $\pathseteq{r}{x}{v}$ is empty if and only if there is
  no vertex $u\in I(x,v) \cap N(v)$ such that
  $\dist{r}{u}=\dist{r}{v}$. If $\pathseteq{r}{x}{v}$ is empty, then
  $\gamma_{\rightarrow}(v) =0$ and we are done.
  
  Suppose now that $\pathseteq{r}{x}{v} \neq \emptyset$.  Let
  $P=(x=v_0,\dots,v_{i-1},v_i=v)$ be a path such that
  $|\coverP{r}{P}| = \gamma_{\rightarrow}(v)$. Observe that
  $\dist{r}{v_{i-1}}=\dist{r}{v_i}$.  Let $Q = (v_0,\dots,v_{i-1})$
  and consider a set $S$ of isometric $r$-rooted paths covering the
  vertices of $Q$ of size $|\coverP{r}{Q}|$ and a $(r,v_i)$-shortest
  path $P_i$. Observe that $S \cup \{P_{i}\}$ is a set of isometric
  $r$-rooted paths covering the vertices of $P$. Consequently
  $\gamma_{\rightarrow}(v_i) = |\coverP{r}{P}| \leq |\coverP{r}{Q}|+1
  \leq \gamma(v_{i-1})+1$.

  Consider now an isometric $(x,v_{i-1})$-path $Q'$ such that
  $\gamma(v_{i-1}) = |\coverP{r}{Q'}|$. Let $P'$ be the isometric
  $(x,v_i)$-path obtained by appending $v_i$ to $Q'$. Consider a set
  $S'$ of isometric $r$-rooted paths covering the vertices of $P'$ of
  size $|\coverP{r}{P'}|$ and let $P_i' $ be a path of $S'$ covering
  $v_i$.  By Lemma~\ref{lemma1}, no vertex of $Q'$ is covered by
  $P_i'$. Consequently, $S'\setminus \{P_i'\}$ is a set of isometric
  $r$-rooted paths covering all vertices of $Q'$ and thus
  $\gamma(v_{i-1}) \leq |\coverP{r}{P'}| -1 \leq
  \gamma_{\rightarrow}(v_i) -1$. Thus, we have
  $\gamma_{\rightarrow}(v_i) = \gamma(v_{i-1}) +1$.
\end{proof}

\begin{lemma}\label{lem:gamma-searrow}
  If $r$ is the root vertex, $x$ the source vertex, and $v$ is a
  vertex distinct from $x$, then {\color{black} $\gamma_{\searrow}(v) = 0$  if $\pathseta{r}{x}{v}=\emptyset$ and} 
  $\gamma_{\searrow}(v) = 
    \max \{ \max\{\gamma_{\searrow}(u),\gamma_{\rightarrow}(u),\beta_{\nearrow}(u)+1\}
  : u\in I(x,v) \cap N(v);\ \dist{r}{u}=\dist{r}{v}-1 \}$ 
  {\color{black} otherwise}.
\end{lemma}

\begin{proof}
  Observe that $\pathseta{r}{x}{v}$ is empty if and only if there is
  no vertex $u\in I(x,v) \cap N(v)$ such that
  $\dist{r}{u}=\dist{r}{v}-1$. If $\pathseta{r}{x}{v}$ is empty, then
  $\gamma_{\searrow}(v) =0$ and we are done.  Assume now that
  $\pathseta{r}{x}{v} \neq \emptyset$. If $v$ is adjacent to $x$, then
  $P = (x,v)$ is the unique isometric $(x,v)$-path, and since
  $\pathseta{r}{x}{v} \neq \emptyset$, we have
  $\dist{r}{x} = \dist{r}{v}-1$.  Then $P$ can be covered by any
  isometric $(r,v)$-path containing $x$, and thus
  $\gamma_{\searrow}(v) = |\coverP{r}{P}| = 1 = \gamma_{\searrow}(x) =
  \gamma_{\rightarrow}(x) = 1 + \beta_{\nearrow}(x)$.
  
  Assume now that $v$ is not adjacent to $x$. Let
  $P=(x=v_0,\dots,v_{i-1},v_i=v)$ be a path such that
  $|\coverP{r}{P}| = \gamma_{\searrow}(v)$, let $Q=(v_0,\dots,v_{i-1})$,
  and let $R=(v_0,\dots,v_{i-2})$. Note that
  $\dist{r}{v_{i-1}}=\dist{r}{v_i}-1$.
  
  First suppose that $\dist{r}{v_{i-2}}=\dist{r}{v_{i-1}}-1$.
  We claim that $|\coverP{r}{P}| \leq |\coverP{r}{Q}|$. Indeed,
  consider a set $S$ of isometric $r$-rooted paths covering the
  vertices of $Q$ of size $|\coverP{r}{Q}|$. Let $P_{i-1} \in S$ be a
  path covering $v_{i-1}$. By Lemma~\ref{lemma1} and since
  $\dist{r}{v_{i-2}}=\dist{r}{v_{i-1}}-1$, we can assume that
  $P_{i-1}$ is an isometric $(r,v_{i-1})$-path. Consider the path
  $P_i$ obtained by appending $v_i$ at the end of $P_{i-1}$ and
  observe that $P_i$ is an isometric $(r,v_{i})$-path covering the
  same vertices as $P_{i-1}$ as well as $v_i$. Consequently, replacing
  $P_{i-1}$ by $P_i$ in $S$, we obtain a set of isometric $r$-rooted
  paths of size $|S| = |\coverP{r}{Q}|$ covering all vertices of $P$,
  establishing that $|\coverP{r}{P}| \leq |\coverP{r}{Q}|$. Since
  $|\coverP{r}{Q}| \leq \gamma_{\searrow}(v_{i-1}) \leq
  \gamma(v_{i-1}) \leq \gamma_{\searrow}(v) = |\coverP{r}{P}| \leq
  |\coverP{r}{Q}|$, we have
  $\gamma_{\searrow}(v) = \gamma_{\searrow}(v_{i-1})$.

  Suppose now that $\dist{r}{v_{i-2}}=\dist{r}{v_{i-1}}$. As in the
  previous case, we show that $|\coverP{r}{P}| \leq
  |\coverP{r}{Q}|$. Indeed, consider a set $S$ of isometric $r$-rooted
  paths covering the vertices of $Q$ of size $|\coverP{r}{Q}|$. Let
  $P_{i-1} \in S$ be a path covering $v_{i-1}$. By Lemma~\ref{lemma1}
  and since $\dist{r}{v_{i-2}}=\dist{r}{v_{i-1}}$, $v_{i-1}$ is the
  unique vertex of $Q$ covered by $P_{i-1}$. Consequently, if we
  replace $P_{i-1}$ in $S$ by an isometric $(r,v_i)$-path going
  through $v_{i-1}$, we obtain a set of isometric $r$-rooted paths of
  size $|S| = |\coverP{r}{Q}|$ covering all vertices of $P$,
  establishing that $|\coverP{r}{P}| \leq |\coverP{r}{Q}|$. Since
  $|\coverP{r}{Q}| \leq \gamma_{\rightarrow}(v_{i-1}) \leq
  \gamma(v_{i-1}) \leq \gamma_{\searrow}(v) = |\coverP{r}{P}| \leq
  |\coverP{r}{Q}|$, we have
  $\gamma_{\searrow}(v) = \gamma_{\rightarrow}(v_{i-1})$.

  Finally, suppose that $\dist{r}{v_{i-2}}=\dist{r}{v_{i-1}}+1$ {\color{black} (see Figure \ref{fig:lemma-gamma-searrow} for an illustration of this case).}
  Consider a set $S$ of isometric $r$-rooted paths covering the
  vertices of $R$ of size $|\coverP{r}{R}|$ and a $(r,v_i)$-shortest
  path $P_i$ containing $v_{i-1}$. Observe that $S \cup \{P_{i}\}$ is
  a set of isometric $r$-rooted paths covering the vertices of
  $P$. Consequently,
  $\gamma_{\searrow}(v_{i}) = |\coverP{r}{P}| \leq |\coverP{r}{R}| + 1
  \leq \gamma(v_{i-2}) +1 \leq \beta_{\nearrow}(v_{i-1})+1$.  Consider
  now an isometric $(x,v_{i-1})$-path $Q'$ such that
  $\beta_{\nearrow}(v_{i-1}) =|\coverP{r}{R'}|$ where
  $R' = Q' - \{v_{i-1}\}$.  Let $P'$ be the isometric $(x,v_i)$-path
  obtained by appending $v_i$ to $Q'$.  Consider a set $S'$ of
  isometric $r$-rooted paths covering the vertices of $P'$ of size
  $|\coverP{r}{P'}|$ and let $P_i'$ be the path of $S'$ covering
  $v_i$.  By Lemma~\ref{lemma1}, the only vertex of $Q'$ that can be
  covered by $P_i'$ is $v_{i-1}$.  Consequently,
  $S' \setminus \{P_i'\}$ is a set of isometric $r$-rooted paths
  covering all vertices of $R'$ and thus
  $\beta_{\nearrow}(v_{i-1}) = |\coverP{r}{R'}| \leq |\coverP{r}{P'}|
  - 1 \leq \gamma_{\searrow}(v_i)-1$. Thus, we have
  $\gamma_{\searrow}(v_i) = \beta_{\nearrow}(v_{i-1})+1$.

  Since the formula for computing $\gamma_{\searrow}(v)$ (given in the
  statement of the lemma) takes into account these three exclusive
  alternatives, it computes $\gamma_{\searrow}(v)$ correctly.
\end{proof}

\begin{figure}[t]
\centering
\includegraphics[scale=0.6]{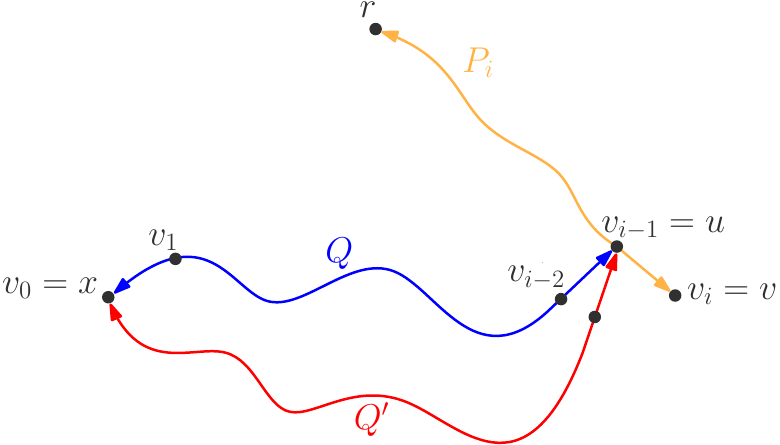}
\caption{\color{black}An illustration for the third case in the proof of Lemma \ref{lem:gamma-searrow}.}
\label{fig:lemma-gamma-searrow}
\end{figure}

\begin{lemma}\label{lem:gamma-nearrow}
  If  $r$ is  the root vertex, $x$  the source vertex, and $v$ is a
  vertex distinct from $x$, then {\color{black} $\gamma_{\nearrow}(v) = 0$ if $\pathsetd{r}{x}{v}=\emptyset$ and}
  $\gamma_{\nearrow}(v) = \max \{
  \max\{\gamma_{\nearrow}(u),\gamma_{\rightarrow}(u),\beta_{\searrow}(u)+1\}
  : u\in I(x,v) \cap N(v);\ \dist{r}{u}=\dist{r}{v}+1 \}$ 
  {\color{black} otherwise}.
\end{lemma}

  \begin{proof}
    The proof is similar to the the proof of
    Lemma~\ref{lem:gamma-searrow}.
  \end{proof}

\sloppy Now we provide a BFS based algorithm to compute the above parameters.
Let $r$ and $x$ be fixed root and source vertices of $G$,
respectively. For a vertex $u\in V(G)$, let
$\mathcal{D}(u) = \{\gamma(u), \gamma_{\nearrow}(u),
\gamma_{\rightarrow}(u), \gamma_{\searrow}(u), \beta_{\nearrow}(u),
\beta_{\searrow}(u)\}$.  Clearly, the set $\mathcal{D}(x)$ can be
computed in constant time. Now let $X_i$ be the set of vertices at
distance $i$ from $x$. Clearly, the sets $X_i$ can be computed in
$O(|E(G)|)$-time (using a BFS) and $X_0=\{x\}$. Let $i\geq 1$ be an
integer and assume that for all vertices
$u \in \bigcup_{j=0}^{i-1} X_j$, the set $\mathcal{D}(u)$ is already
computed. Let $v\in X_i$ be a vertex. Then due to the formulas given
in Observation~\ref{lem:beta-arrow} and
Lemmas~\ref{lem:gamma-rightarrow}--\ref{lem:gamma-nearrow}, the set
$\mathcal{D}(v)$ can be computed by observing only the sets
$\mathcal{D}(u)$, $u\in N(v) \cap X_{i-1}$. Hence, for all vertices
$v\in V(G)$, the sets $\mathcal{D}(v)$ can be computed in a total of
$O(|E(G)|)$ time. Hence, we have the following lemma.

\begin{lemma}\label{lem:main-gamma}
    For a root vertex $r$ and source vertex $x$, for each vertex $v\in V(G)$, the value $\gamma^r(x,v)$ can be computed in $O(|E(G)|)$ time. 
\end{lemma}


We can now finish the proof of Theorem~\ref{thm:ipcoInP}.  Let $G$ be
a graph with $n$ vertices and $m$ edges. For a root vertex $r$, by
applying Lemma~\ref{lem:main-gamma}, for every source $x \in V(G)$, it
is possible to compute
$\ipco{\overrightarrow{G_r}} = \max_{x,v} \gamma^r(x,v)$ in $O(nm)$
time. By repeating this for every root $r\in V(G)$, it is possible to
compute $\ipco{G} = \min_r \ipco{\overrightarrow{G_r}}$ in $O(n^2 m)$
time.

\section{Proof of Theorem~\ref{thm:main}} \label{sec:main}

\newcommand{\Gromov}[3]{\left(#1|#2\right)_{#3}}
\textbf{First we prove Theorem~\ref{thm:main}\ref{thm:hyperbolicity}.} We recall the definition of Gromov products~\cite{gromov1987} and its relation with hyperbolicity. For three vertices $r,x,y$ of a graph $G$, the Gromov product of $x,y$ with respect to $r$ is defined as $\Gromov{x}{y}{r} = \frac{1}{2}\left(\dist{x}{r} + \dist{y}{r} - \dist{x}{y}\right)$.
Then, a graph $G$ is $\delta$-hyperbolic~\cite{chepoi2019fast,gromov1987} if and only if for any four vertices $x,y,z,r$, we have $\Gromov{x}{y}{r} \geq \min\left\{ \Gromov{x}{z}{r}, \Gromov{y}{z}{r} \right\} - \delta$.


Let $G$ be a graph with hyperbolicity at most $\delta$. Due to
Lemma~\ref{lem:ipac-ipco}, in order to prove
Theorem~\ref{thm:main}\ref{thm:hyperbolicity}, it is enough to show
that $\ipac{G} \leq 4\delta+3$.
Aiming {for a} contradiction, let $r$ be a vertex of $G$ and $P$ be an
isometric path such that $|\anticp{r}{P}|\geq 4\delta+4$. Let
$a_1, a_2, \ldots, a_{2\delta+2}, \ldots, a_{4\delta+4}$ be the
vertices of $\anticp{r}{P}$ ordered as they are encountered while
traversing $P$ from one endpoint to the other. Let
$x=a_1, z=a_{2\delta+2}, y=a_{4\delta+4}.$ Let $Q$ denote the
$(y,z)$-subpath of $P$. Observe that,
$|\anticp{r}{Q}| \geq 2\delta+2$.  Then we have 
$ \Gromov{x}{y}{r} \geq \min\left\{ \Gromov{x}{z}{r}, \Gromov{y}{z}{r}
\right\} - \delta$. Without loss of generality, assume that
$\Gromov{x}{z}{r} \leq \Gromov{y}{z}{r}$. Hence,
\begin{align*}
  \Gromov{x}{y}{r} & \geq \Gromov{x}{z}{r} - \delta \\
  \dist{x}{r} + \dist{y}{r} - \dist{x}{y} & \geq \dist{x}{r} + \dist{z}{r} - \dist{x}{z} - 2\delta\\
  \dist{y}{r} - \dist{x}{y} & \geq \dist{z}{r} - \dist{x}{z} - 2\delta\\
  \dist{y}{r} - \dist{z}{r} + 2\delta & \geq \dist{x}{y} - \dist{x}{z}\\
  \dist{y}{r} - \dist{z}{r} + 2\delta & \geq \dist{y}{z}\\
  \dist{y}{z} & \leq \left|\dist{y}{r} - \dist{z}{r}\right| + 2 \delta.
\end{align*}

But this directly contradicts Proposition~\ref{prp:antichain-length}, which implies that $\dist{y}{z} \geq \left|\dist{y}{r} - \dist{z}{r}\right| + \left|\anticp{r}{Q}\right| - 1 \geq \left|\dist{y}{r} - \dist{z}{r}\right| + 2\delta + 1 $. This completes the proof of Theorem~\ref{thm:main}\ref{thm:hyperbolicity}.

\medskip\noindent
\textbf{Now, we shall prove Theorems~\ref{thm:main}\ref{thm:truemper} and \ref{thm:main}\ref{thm:outer}.} First, we shall define the notions of \emph{$t$-theta}, \emph{$t$-prism}, and \emph{$t$-pyramid}~\cite{trotignonprivate}. For an integer $t\geq 1$, a \emph{$t$-prism} is a graph made of three vertex-disjoint induced paths $P_1 = a_1\ldots b_1$, $P_2 = a_2\ldots b_2$, $P_3 = a_3\ldots b_3$ of lengths at least $t$, such that $a_1 a_2 a_3$ and $b_1 b_2b_3$ are triangles and no edges exist between the paths except those of the two triangles. For an integer $t\geq 1$, a \emph{$t$-pyramid} is a graph made of three induced paths $P_1 = a\ldots b_1$, $P_2 = a\ldots b_2$, $P_3 = a\ldots b_3$ of lengths at least $t$, two of which
 have lengths at least $t+1$, {they are pairwise} vertex-disjoint except at $a$, such that $b_1 b_2 b_3$ is a triangle and no edges exist between the paths except those of the triangle and the three edges incident to $a$. For an integer $t\geq 1$, a \emph{$t$-theta} is a graph made of three {internally} vertex-disjoint induced paths $P_1 = a\ldots b$, $P_2 = a\ldots b$, $P_3 = a\ldots b$ of lengths at least $t+1$, and such that no edges exist between the paths except the three edges incident to $a$ and the three edges incident to $b$. A graph $G$ is \emph{($t$-theta, $t$-pyramid, $t$-prism)}-free if $G$ does not contain any induced subgraph isomorphic to a $t$-theta, $t$-pyramid or $t$-prism. When $t=1$, ($t$-theta, $t$-pyramid, $t$-prism)-free graphs are exactly (theta, prism, pyramid)-free graphs.

Now, we shall show that the isometric path antichain cover number of ($t$-theta, $t$-pyramid, $t$-prism)-free graphs are bounded above by a linear function on $t$. We shall show that, when the isometric path antichain cover number of a graph is large, the existence of a structure called ``$t$-fat turtle'' (defined below) as an induced subgraph is forced, which, cannot be present in a ($(t-1)$-theta, $(t-1)$-pyramid, $(t-1)$-prism)-free graph. 


\begin{figure}
    \centering
    
    \begin{tikzpicture}[scale=0.8]

    \filldraw (-1,1) circle (2pt);
    \filldraw (-1,-1) circle (2pt);
    \filldraw (-1.35,0) circle (2pt);
    
    \draw[thick] (-1,1) -- (-1.35,0) -- (-1,-1);
    \draw[thick] (0,0) -- (-1,1);
    \draw[thick] (0,0) -- (-1,-1);
    \draw[thick] (5,1) -- (-1,1);
    \filldraw (1,1) circle (2pt);
    \filldraw (2,1) circle (2pt);
    \filldraw (3,1) circle (2pt);
    
    \filldraw (0.5,-1) circle (2pt);
    \filldraw (1.5,-1) circle (2pt);
    \filldraw (2.5,-1) circle (2pt);
    \filldraw (3.5,-1) circle (2pt);
    
    \filldraw (5,1) circle (2pt);
    \filldraw (5,-1) circle (2pt);
    \filldraw (5.35,0) circle (2pt);
    
    \draw[thick] (5,1) -- (5.35,0) -- (5,-1);
    \draw[thick] (4,0) -- (5,1);
    \draw[thick] (4,0) -- (5,-1);
    \draw[thick] (4,0) -- (5.35,0);
    
    \draw[thick] (5,-1) -- (-1,-1);
    
    \filldraw (4,0) circle (2pt);
    
    \draw[thick] (0,0) -- (4,0);
        
    \draw[fill=white, draw=black] (2,0) circle (2pt); 
    \draw[fill=white, draw=black] (3,0) circle (2pt);
    \draw[fill=white, draw=black] (0,0) circle (2pt);\node[above] at (0,0.1) {$u$};
    \draw[fill=white, draw=black] (1,0) circle (2pt);
    \draw[fill=white, draw=black] (4,0) circle (2pt);\node[above] at (4,0.1) {$v$};

    \node[above] at (0.5,-0.9) {$c$};
    \node[below] at (1,0.9) {$c'$};

    \end{tikzpicture}
    \caption{An example of a $4$-fat turtle. Let $C$ be the cycle induced by the black vertices, $P$ be the path induced by the white vertices. Then the tuple $(4,C,P,c,c')$ {defines} a $4$-fat turtle.
    }
    \label{fig:turtle}
\end{figure}

\begin{definition}
For an integer $t\geq 1$, a ``$t$-fat turtle'' consists of a cycle $C$ and an induced $(u,v)$-path $P$ of length at least $t$ 
such that all of the following hold: 
\begin{enumerate}[label=(\alph*)]

\item $V(P) \cap V(C) = \emptyset$,

\item For any vertex $w\in (V(P)\setminus \{u,v\})$, $N(w) \cap V(C) = \emptyset$ and both $u$ and $v$ have at least one neighbour in $C$,

\item For any vertex $w\in N(u) \cap V(C)$ and $w'\in N(v)\cap V(C)$, the distance between $w$ and $w'$ in $C$ is at least $t$,

\item There exist two vertices $\{c,c'\}\subset V(C)$ and two distinct components $C_u,C_v$ of $C-\{c,c'\}$ such that $N(u) \cap V(C) \subseteq V(C_u)$ and $N(v) \cap V(C) \subseteq V(C_v)$.

\end{enumerate}

The tuple $(t,C,P,c,c')$ denotes the $t$-fat turtle. See Figure~\ref{fig:turtle} for an example.
\end{definition}

In the following observation, we show that any ($t$-theta, $t$-pyramid,$t$-prism)-free graph cannot contain a $(t+1)$-fat turtle as an induced subgraph.

\begin{lemma}\label{lem:fat-turtle}
For some integer $t\geq 1$, let $G$ be a graph containing a $(t+1)$-fat turtle as an induced subgraph. Then $G$ is not ($t$-theta, $t$-pyramid, $t$-prism)-free.  
\end{lemma}
 \begin{proof}
 Let $(t+1,C,P,c,c')$ be a $(t+1)$-fat turtle in $G$. Let the vertices of $C$ be named $c=a_0, a_1, \ldots, a_k=c', a_{k+1},\ldots, a_{|V(C)|}$ as they are encountered while traversing $C$ starting from $c$ in a counter-clockwise manner. Denote by
 $u,v$ the endpoints of $P$. By definition, there exist two distinct components $C_u,C_v$ of $C-\{c,c'\}$ such that $N(u) \cap V(C) \subseteq V(C_u)$ and $N(v) \cap V(C) \subseteq V(C_v)$. Without loss of generality, assume $V(C_u) = \{a_1, a_2, \ldots, a_{k-1}\}$ and $V(C_v) = \{a_{k+1}, a_{k+2}, \ldots, a_{|V(C)|}\}$. Let $i^-$ and $i^+$ be the minimum and maximum indices such that $a_{i^-}$ and $a_{i^+}$ are adjacent to $u$. Let $j^-$ and $j^+$ be the minimum and maximum indices such that $a_{j^-}$ and $a_{j^+}$ are adjacent to $v$. By definition, $i^-\leq i^+ < j^-\leq j^+$. Let $P_1$ be the $(a_{i^-},a_{j^+})$-subpath of $C$ containing $c$. Let $P_2$ be the $(a_{i^+},a_{j^-})$-subpath of $C$ that contains $c'$. Observe that $P_1$ and $P_2$ have length at least $t$ (by definition). Now we show that $P,P_1,P_2$ together form one of theta, pyramid or prism. If $a_{i^-} = a_{i^+}$ and $a_{j^-} = a_{j^+}$, then $P,P_1,P_2$ form a $t$-theta. If $i^-\leq i^+-2$ and $j^-\leq j^+-2$, then also $P,P_1,P_2$ form a $t$-theta. If $j^-= j^+-1$ and $i^-= i^+-1$, then $P,P_1,P_2$ form a $t$-prism. In any other case, $P,P_1,P_2$ form a $t$-pyramid.
 \end{proof}

In the remainder of this section, we shall prove that there exists a linear function $f(t)$ such that if the isometric path antichain cover number of a graph is more than $f(t)$, then $G$ is forced to contain a $(t+1)$-fat turtle as an induced subgraph, and therefore is not ($t$-theta, $t$-pyramid,$t$-prism)-free. We shall use the following observation.



\begin{observation}\label{obs:anti-cp-reduc}
\sloppy Let $G$ be a graph, $r$ be an arbitrary vertex, $P$ be an isometric $(u,v)$-path in $G$ and $Q$ be a subpath of an isometric $(v,r)$-path in $G$ such that one endpoint of $Q$ is $v$. Let $P'$ be the maximum $(u,w)$-subpath of $P$ such that no internal vertex of $P'$ is a neighbour of some vertex of $Q$. We have that $|\anticp{r}{P'}| \geq |\anticp{r}{P}| - 3$. 
\end{observation}

 \begin{proof}
 Suppose $|\anticp{r}{P'}| \leq |\anticp{r}{P}| - 4$ and consider the $(w,v)$-subpath, say $P''$, of $P$. Observe that $|\anticp{r}{P''}| \geq 4$. Now let $w'$ be a vertex of $Q$ which is a neighbour of $w$. Observe that $|\dist{r}{w} - \dist{r}{w'}| \leq 1$ and therefore $\dist{w}{v} = |E(P'')| \leq |\dist{r}{w} - \dist{r}{v}| + 2 $. But this contradicts Proposition~\ref{prp:antichain-length}, which implies that the length of $P''$ is at least $|\dist{r}{w} - \dist{r}{v}| + 3$.
 \end{proof}


 
    
     



\newcommand{\vertical}[2]{Q\left({#1},{#2}\right)}

\begin{figure}[t]
\centering
    \scalebox{0.9}{
    \centering
    \begin{tikzpicture}
    
    \filldraw (3.5,4) circle (2pt);
    \node[above] at (3.5,4) {$r$};
    \filldraw (3.5,2) circle (2pt);
    \filldraw (5,2.5) circle (2pt);
    \filldraw (2,2.5) circle (2pt);
    \node[left] at (2,2.6) {$z_2$};
    \node[above] at (5.15,2.45) {$w_2$};

    \draw[thick] (2,2.5) -- (5,2.5);
    \draw[thick] (3.5,2) -- (5,2.5);

    \filldraw (0,0) circle (2pt);
    \node[below] at (0,0) {$u$};
    
    \filldraw (0.5,0) circle (2pt);
    \filldraw (1,0) circle (2pt);
    \node[right] at (1.1,0) {$z$};
    
    \draw[thick] (0,0) -- (1,0);
    \draw[thick] (0,0) -- (0.5,1) -- (1,0);

    \filldraw (0.5,1) circle (2pt);
    \node[left] at (0.5,1) {$z_1$};

    \filldraw (6,0) circle (2pt);
    \node[left] at (6,0) {$w$};
    \filldraw (6.5,0) circle (2pt);
    \filldraw (7,0) circle (2pt);
    \node[right] at (7,0) {$b$};
    
    \draw[thick] (6,0) -- (7,0);
    \node[right] at (6.5,1) {$w_1$};
    \draw[thick] (6,0) -- (6.5,1) -- (7,0);
    
    \filldraw (6.5,1) circle (2pt);
    \filldraw (3.5,-1.5) circle (2pt);

    \filldraw (3,-2) circle (2pt);
    \filldraw (3.5,-2) circle (2pt);
    \filldraw (4,-2) circle (2pt);
    \node[below] at (3,-2) {\begin{tabular}{c}
         $c$ \\ $(= a^{}_{2t+13})$ 
    \end{tabular}};
    \node[below] at (4,-2) {\begin{tabular}{c} $x$ \end{tabular}};
    
    \draw[thick] (3,-2) -- (4,-2);
    
    \draw[thick] (3,-2) -- (3.5,-1.5) -- (4,-2) ;
    \node[right] at (3.3,-1.5) {\begin{tabular}{c}
         $c_1$ 
    \end{tabular}};
    
    \filldraw (-4,-2) circle (2pt);
    \node[below] at (-4,-2) {$a$};
    \node[right] at (3.3,1.95) {\begin{tabular}{c}
         $c_2$ 
    \end{tabular}};
    
    \node[right] at (3.3,0) {\begin{tabular}{c}
         \scriptsize $T(c_1,c_2)$ 
    \end{tabular}};
    \node[left] at (3.6,0) {\begin{tabular}{c}
         \scriptsize $\geq t$ 
    \end{tabular}};

    \node[left] at (6.1,-1) {\begin{tabular}{c}
         \scriptsize $\geq t$ 
    \end{tabular}};

    \node[left] at (3.1,-1) {\begin{tabular}{c}
         \scriptsize $\geq t$ 
    \end{tabular}};
    
    \path [draw=black, very thick,snake it] (-4,-2) -- (0,0);
    \path [draw=black, very thick,snake it] (1,0) -- (3,-2);
    \path [draw=black, very thick,snake it] (6,0) -- (4,-2);
    \path [draw=black, thick,snake it] (0.5,1) -- (3.5,4);
    \path [draw=black, thick,snake it] (6.5,1) -- (3.5,4);
    \path [draw=black, thick, snake it] (3.5,-1.5) -- (3.5,2) -- (3.5,4);
    
    \node[right] at (5.5,2) {\begin{tabular}{c}
          $\vertical{r}{b}$ 
    \end{tabular}};
    
    \node[right] at (1, 1.6) {\begin{tabular}{c}
         $\vertical{r}{u}$ 
    \end{tabular}};
    
    \end{tikzpicture}}
    \caption{Illustration of the notations used in the proof of Lemma~\ref{lem:ipac-fat}. The thick wiggly line along with the horizontal edges indicate the $(a,b)$-isometric path $P$. }\label{fig:proof-illustrate}
\end{figure}

\begin{lemma}\label{lem:ipac-fat}
For an integer $t\geq 1$, let $G$ be a graph with $\ipac{G}\geq 8t+64$. Then $G$ has a $(t+1)$-fat turtle {as an induced subgraph}.
\end{lemma}

\begin{proof}

Let $r$ be a vertex of $G$ such that $\ipac{\overrightarrow{G_r}}$ is at least $8t+64$. Then there exists an isometric path $P$ such that $|\anticp{r}{P}|\geq 8t+64$. Let the two endpoints of $P$ be $a$ and $b$. (See Figure~\ref{fig:proof-illustrate}.) Let $u$ be a vertex of $P$ such that $\dist{r}{u}=\dist{r}{P}$. Let $\Pnote{a}{u}$ be the $(a,u)$-subpath of $P$ and $\Pnote{b}{u}$ be the $(b,u)$-subpath of $P$. Both $\Pnote{a}{u}$ and $\Pnote{b}{u}$ are isometric paths and observe that either $|\anticp{r}{\Pnote{a}{u}}|\geq 4t+32$ or  $|\anticp{r}{\Pnote{b}{u}}|\geq 4t+32$. Without loss of generality, assume that  $|\anticp{r}{\Pnote{b}{u} }|\geq 4t+32$. Let $\vertical{r}{b}$ be an isometric $(b,r)$-path in $G$. First observe that $u$ is not adjacent to any vertex of $\vertical{r}{b}$. Otherwise, $\dist{u}{b} \leq 2 + \dist{r}{b}-\dist{r}{u}$, which contradicts Proposition~\ref{prp:antichain-length}. Let $\Pnote{u}{w}$ be the maximum $(u,w)$-subpath, of $\Pnote{b}{u}$ such that no internal vertex of $\Pnote{u}{w}$ is a neighbour of $\vertical{r}{b}$. Note that $\Pnote{u}{w}$ is an isometric path and $w$ has a neighbour in $\vertical{r}{b}$. Applying Observation~\ref{obs:anti-cp-reduc}, we have the following:

\begin{claim}
$|\anticp{r}{\Pnote{u}{w}}| \geq 4t+29$.
\end{claim}

Let $\vertical{r}{u}$ be any isometric $(u,r)$-path of $G$. Observe that $w$ is not adjacent to any vertex of $\vertical{r}{u}$. Otherwise, $\dist{u}{w} \leq 2 + \dist{r}{u}-\dist{r}{w}$, which contradicts Proposition~\ref{prp:antichain-length}.  Let $\Pnote{z}{w}$ be the maximum $(z,w)$-subpath of $\Pnote{u}{w}$ such that no internal vertex of $\Pnote{z}{w}$ has a neighbour in $\vertical{r}{u}$. Observe that $\Pnote{z}{w}$ is an isometric path, and $z$ has a neighbour in $\vertical{r}{u}$. Again applying Observation~\ref{obs:anti-cp-reduc}, we have the following:

\begin{claim}\label{cl:2}
$|\anticp{r}{\Pnote{z}{w}}| \geq 4t+26$.
\end{claim}

Let $a_1, a_2, \ldots,a_k$ be the vertices of $\anticp{r}{\Pnote{z}{w}}$ ordered according to their appearance while traversing $\Pnote{z}{w}$ from $z$ to $w$. Due to Claim~\ref{cl:2}, we have that $k\geq 4t+26$. Let $c=a_{2t+13}$ and $\vertical{r}{c}$ denote an isometric $(c,r)$-path. Let $T(r,c_1)$ 
be the maximum subpath of $\vertical{r}{c}$ such that no internal vertex of $T(r,c_1)$ is adjacent to any vertex of $\Pnote{z}{w}$. Observe that neither $z$ nor $w$ can be adjacent to $c_1$ (due to Proposition~\ref{prp:antichain-length}). Morevoer, if $c_1$ is a vertex of $\Pnote{z}{w}$ then we must have $c_1=c$.

\begin{claim}\label{clm:life-saver}
Let $x$ be a neighbour of $c_1$ in $\Pnote{z}{w}$, $X$ be the $(x,b)$-subpath of $\Pnote{u}{b}$ and $Y$ be the $(x,u)$-subpath of $\Pnote{u}{b}$. Then $|\anticp{r}{X}| \geq 2t+11$ and $|\anticp{r}{Y}| \geq 2t+11$.
\end{claim}

 \begin{subproof}
\sloppy  Let $\Pnote{c}{w}$ denote the 
 $(c,w)$-subpath of $\Pnote{z}{w}$. Observe that $|\anticp{r}{\Pnote{c}{w}}| \geq 2t+14$.
 First, consider the case when $x$ lies in the $(z,c)$-subpath of $\Pnote{z}{w}$. In this case, $\Pnote{c}{w}$ is a subpath of $X$ and therefore $|\anticp{r}{X}| \geq 2t+14$. Now consider the case when $x$ lies in $\Pnote{c}{w}$. In this case, applying Observation~\ref{obs:anti-cp-reduc}, we have that $|\anticp{r}{X}| \geq |\anticp{r}{\Pnote{c}{w}}| - 3 \geq 2t+11$. Using a similar argument, we have that $|\anticp{r}{Y}| \geq 2t+11$.
 \end{subproof}

Let $T(c_1,c_2)$ be the maximum $(c_1,c_2)$-subpath of $T(c_1,r)$ such that no internal vertex of $T(c_1,c_2)$ is adjacent to a vertex of $\vertical{r}{b}$ or $\vertical{r}{u}$. {Note that, if $c_2$ lies on $\vertical{r}{b}$ or $\vertical{r}{u}$, we must have $c_2=r$}.  We have the following claim.

 \begin{claim}\label{clm:vacant-path}
 The length of $T(c_1,c_2)$ is at least $t+3$.
 \end{claim}

  \begin{subproof}
 Assume that the length of $T(c_1,c_2)$ is at most $t+2$ and $x$ be a neighbour of $c_1$ in $\Pnote{z}{w}$. Observe that all vertices of $\Pnote{z}{w}$ are at distance at least $\dist{r}{u}$ \textit{i.e.} $\dist{r}{\Pnote{z}{w}} \geq \dist{r}{u}$, {since $\dist{r}{u}=\dist{r}{P}$}. Hence, 

 \medskip \noindent \textbf{(+)} $\dist{r}{x} \geq \dist{r}{u}$ and $\dist{r}{c_1} \geq \dist{r}{u} - 1$.

\medskip \noindent Now, suppose $c_2$ has a neighbour $c_3$ in $\vertical{r}{u}$. Hence $\dist{c_3}{x} \leq \dist{c_3}{c_2} + \dist{c_2}{c_1} + \dist{c_1}{x} \leq t+4$. Now, using (+) and the fact that $c_3$ lies on an isometric $(r,u)$-path ($\vertical{r}{u}$), we have that $\dist{c_3}{u} \leq t+4$. Therefore, $\dist{u}{x} \leq \dist{c_3}{u} + \dist{c_3}{x} \leq 2t+8$.  But this contradicts  Proposition~\ref{prp:antichain-length} and Claim~\ref{clm:life-saver}, as they together imply that $\dist{u}{x}$ is at least $\dist{r}{x} - \dist{r}{u} + 2t+10 {\geq 2t+10}$.

 Hence, $c_2$ must have a neighbour $c_3$ in $\vertical{r}{b}$. First, assume that $\dist{r}{x} \geq \dist{r}{b}$. Then, as $\dist{c_3}{x} \leq \dist{c_3}{c_2} + \dist{c_2}{c_1} + \dist{c_1}{x} \leq t+4$ and $c_3$ lies on an isometric $(r,b)$-path ($\vertical{r}{b}$), we have that $\dist{x}{b} \leq 2t+8$. But {again} this contradicts  Proposition~\ref{prp:antichain-length} and Claim~\ref{clm:life-saver}, as they together imply that the length of $\dist{x}{b}$ is at least $\dist{r}{x} - \dist{r}{u} + 2t+10$. Now, assume that $\dist{r}{x} < \dist{r}{b}$. Let $b'$ be a vertex of $\vertical{r}{b}$ such that $\dist{r}{b'} = \dist{r}{x}$. Using a similar argumentation as before, we have that $\dist{x}{b'} \leq 2t+8$. Hence, $\dist{x}{b} \leq \dist{x}{b'} + \dist{b'}{b} \leq \dist{r}{b} - \dist{r}{x} + 2t+8$.
 But this contradicts Proposition~\ref{prp:antichain-length} which, due to Claim~\ref{clm:life-saver}, implies that $\dist{x}{b} \geq \dist{r}{b} - \dist{r}{x} + 2t+10$.
 \end{subproof}

The path $T(c_1,c_2)$ forms the first ingredient to extract a $(t+1)$-fat turtle. Let $z_1$ be the neighbour of $z$ in $\vertical{r}{u}$ and $w_1$ be the neighbour of $w$ in $\vertical{r}{b}$. We have the following claim.

\begin{claim}\label{clm:final}
The vertices $w_1$ and $z_1$ are non adjacent.
\end{claim}

 \begin{subproof}
 Recall that $z_1$ lies in $\vertical{r}{u}$ and $\dist{r}{z} \geq \dist{r}{u}$. Hence $z_1$ must be a neighbour of $u$. If $w_1$ and $z_1$ are adjacent, then observe that $\dist{u}{b} \leq \dist{r}{b} - \dist{r}{w_1} + 2 \leq $. This implies $\dist{u}{b} \leq \dist{r}{b} - \dist{r}{u} + 3$. But this shall again contradict Proposition~\ref{prp:antichain-length}.
 \end{subproof}

Now we shall construct a $(w_1,z_1)$-path as follows: Consider the maximum $(w_1,w_2)$-subpath, say $T(w_1,w_2)$, of $\vertical{r}{b}$ such that no internal vertex of $T(w_1,w_2)$ has a neighbour in $\vertical{r}{u}$. Similarly, consider the maximum $(z_1,z_2)$-subpath, say $T(z_1,z_2)$, of $\vertical{r}{u}$ such that no internal vertex of $T(z_1,z_2)$ is a neighbour of $w_2$. {(Note that it is possible that $z_2=w_2=r$.)}
Let $T$ be the path obtained by taking the union of $T(w_1,w_2)$ and $T(z_1,z_2)$. Observe that $z_2$ must be a neighbour of $w_2$ and $T$ is an induced $(w_1,z_1)$-path. The definitions of $T$ and $\Pnote{z}{w}$ imply that their union induces a cycle $Z$. Here we have the second and final ingredient to extract the $(t+1)$-fat turtle.

Suppose that $c_2$ has a neighbour in $T$. Let $T'$ be the maximum subpath of $T(c_1,c_2)$ which is vertex-disjoint
from $Z$. ({Note that if $c_1=c$ or $c_2\in \{w_2,z_2\}$ (e.g. when $c_2=w_2=z_2=r$), $T(c_1,c_2)$ may share vertices with $Z$}.) Due to Claim~\ref{clm:vacant-path}, the length of $T'$ is at least $t+1$. Let $e_1$ and $e_2$ be the endpoints of $T'$. Observe the following.

\begin{itemize}
    \item Each of $e_1$ and $e_2$ has at least one neighbour in $Z$.
    
    \item $Z-\{z,w\}$ contains two distinct components $C_1,C_2$ such that for $i\in \{1,2\}$, $N(e_i)\cap V(Z) \subseteq V(C_i)$.
    
    \item For a vertex $e_1'\in N(e_1)\cap V(Z)$ and $e_2'\in N(e_2)\cap V(Z)$, the distance between $e'_1$ and $e'_2$ is at least $t+1$. This statement follows from Claim~\ref{clm:life-saver}.
\end{itemize}

Hence, we have that the tuple $(t+1,Z,T',z,w)$ defines a $(t+1)$-fat turtle. Now consider the case when $c_2$ does not have a neighbour in $T$. By definition, $c_2$ has at least one neighbour in $\vertical{r}{u}$ or $\vertical{r}{b}$. Without loss of generality, assume that $c_2$ has a neighbour $c_3$ in $\vertical{r}{u}$ such that the $(z_2,c_3)$-subpath, say, $T''$ of $\vertical{r}{u}$ has no neighbour of $c_2$ {other than $c_3$}. Observe that the path $T^*= (T' \cup (T''-\{z_2\}))$ is vertex-disjoint from $Z$ and has length at least $t+1$. Let $e_1,e_2$ be the two endpoints of $T^*$. Observe the following.
\begin{itemize}
    \item Each of $e_1$ and $e_2$ has at least one neighbour in $Z$.
    
    \item $Z-\{z,w\}$ contains two distinct components $C_1,C_2$ such that for $i\in \{1,2\}$, $N(e_i)\cap V(Z) \subseteq V(C_i)$.
    
    \item For a vertex $e_1'\in N(e_1)\cap V(Z)$ and $e_2'\in N(e_2)\cap V(Z)$, the distance between $e'_1$ and $e'_2$ is at least $t+1$. This statement follows from Claim~\ref{clm:life-saver}.
\end{itemize}

 Hence, $(t+1,Z,T^*,z,w)$ is a $(t+1)$-fat turtle
\end{proof}

\noindent \textbf{Proof of Theorem~\ref{thm:main}\ref{thm:truemper}:} Lemma~\ref{lem:ipac-ipco}, \ref{lem:fat-turtle} and~\ref{lem:ipac-fat} together imply Theorem~\ref{thm:main}\ref{thm:truemper}.



 \begin{lemma}\label{lem:outer-truemper}
 Any outerstring graph is ($4$-theta, $4$-prism, $4$-pyramid)-free.
 \end{lemma}

  \begin{proof}
 To prove the lemma, we shall need to recall a few definitions and results from the literature. A graph $G$ is a \emph{string} graph if there is a collection $S$ of simple curves on the plane and a bijection between $V(G)$ and $S$ such that two curves in $S$ intersect if and only if the corresponding vertices are adjacent in $G$. Let $G$ be a graph with an edge $e$. The graph $G/ e$ is obtained by \emph{contracting} the edge $e$ into a single vertex. Observe that string graphs are closed under edge contraction~\cite{kratochvil1991string}. We shall use the following result.

 \begin{proposition}[\cite{kratochvil1991string}]\label{prp:contract}
 Let $G$ be an outerstring graph with an edge $e$. Then $G/ e$ is an outerstring graph.
 \end{proposition}

  A \emph{full subdivision} of a graph is a graph obtained by replacing each edge of $G$ with a new path of length at least~2.  We shall use the following result implied from Theorem $1$ of~\cite{kratochvil1991string}. 

 \begin{proposition}[\cite{kratochvil1991string}]\label{prp:k33}
 Let $G$ be a string graph. Then $G$ does not contain a full subdivision of $K_{3,3}$ as an induced subgraph.
 \end{proposition}

\textcolor{black}{Indeed, if the full subdivision of $K_{3,3}$ was a string graph, then the corresponding intersection representation with simple curves could be used to construct a planar drawing of $K_{3,3}$, which is impossible.} For a graph $G$, the graph $G^+$ is constructed by introducing a new \emph{apex} vertex $a$ and connecting $a$ with all vertices of $G$ by new copies of paths of length at least $2$. We shall use the following result of Biedl \textit{et al.}~\cite{biedl2018size}.

 \begin{proposition}[Lemma 1, \cite{biedl2018size}] \label{prp:apex-string}
 A graph $G$ is an outerstring graph if and only if $G^+$ is a string graph.
 \end{proposition}

 Now we are ready to prove the lemma. Let $G$ be an outerstring graph. Assume for the sake of contradiction that $G$ contains an induced subgraph $H$ which is a $4$-theta, $4$-pyramid, or a $4$-prism. Since every induced subgraph of an outerstring graph is also an outerstring graph, we have that $H$ is an outerstring graph. Let $E$ be the set of edges of $H$ whose both endpoints are part of the \textcolor{black}{same} triangle. Now consider the graph $H_1= H / E$ which is obtained by contracting all edges in $E$. By Proposition~\ref{prp:contract}, $H_1$ is an outerstring graph and it is easy to check that $H_1$ is a $3$-theta.
 Let $u$ and $v$ be the vertices of $H_1$ with degree~3 and $w_1,w_2,w_3$ be the set of mutually non-adjacent vertices such that for each $i\in \{1,2,3\}$ $\dist{u}{w_i}=2$ and $\dist{v}{w_i}\geq 2$. Since $H_1$ is a $3$-theta, $w_1,w_2,w_3$ exist. Now consider the graph $H_1^+$ and let $a$ be the new apex vertex. Due to Proposition~\ref{prp:apex-string}, we have that $H_1^+$ is a string graph. But notice that, for each pair of vertices in $\{x,y\} \subset \{w_1,w_2,w_3,u,v,a\}$, there exists a unique path of length {at least 2} connecting $x,y$. This implies that $H_1^+$ (which is a string graph) contains a full subdivision of $K_{3,3}$, which contradicts Proposition~\ref{prp:k33}. 
  \end{proof}

\noindent \textbf{Proof of Theorem~\ref{thm:main}\ref{thm:outer}:} Lemma~\ref{lem:ipac-ipco}, \ref{lem:fat-turtle}, \ref{lem:ipac-fat}, and~\ref{lem:outer-truemper} together imply Theorem~\ref{thm:main}\ref{thm:outer}.

\section{Proof of Theorem~\ref{thm:lower}}\label{sec:lower}

\newcommand{\Graph}[1]{X_{#1}}
\newcommand{\BGraph}[1]{Y_{#1}}
\newcommand{\CGraph}[1]{Z_{#1}}
\newcommand{\DGraph}[1]{W_{#1}}

\begin{figure}
    \centering
    \scalebox{1}{
    \begin{tabular}{cccc}
       \begin{tikzpicture}
       \filldraw (0,0) circle (2pt);
       \filldraw (0,-0.5) circle (2pt);
       \filldraw (0,-1) circle (2pt);
       \filldraw (0,-1.5) circle (2pt);
       \filldraw (0,-2) circle (2pt);
       
    \filldraw (-0.5,-0.5) circle (2pt);
    \filldraw (-0.5,-1) circle (2pt);
    \filldraw (-0.5,-1.5) circle (2pt);
    \filldraw (-0.5,-2) circle (2pt);
    
    \filldraw (-1,-0.5) circle (2pt);
    \filldraw (-1,-1) circle (2pt);
    \filldraw (-1,-1.5) circle (2pt);
    \filldraw (-1,-2) circle (2pt);
    
    \filldraw (0.5,-0.5) circle (2pt);
    \filldraw (0.5,-1) circle (2pt);
    \filldraw (0.5,-1.5) circle (2pt);
    \filldraw (0.5,-2) circle (2pt);
    
    \filldraw (1,-0.5) circle (2pt);
    \filldraw (1,-1) circle (2pt);
    \filldraw (1,-1.5) circle (2pt);
    \filldraw (1,-2) circle (2pt);
    
    \draw[thick] (0,0) -- (0.5,-0.5) -- (0.5,-2);
    \draw[thick] (0,0) -- (-0.5,-0.5) -- (-0.5,-2);
    \draw[thick] (0,0) -- (-1,-0.5) -- (-1,-2);
    \draw[thick] (0,0) -- (1,-0.5) -- (1,-2);
    \draw[thick] (0,0) -- (0,-0.5) -- (0,-2);
    \draw[thick] (-1,-2) -- (1,-2);
    
    \node[above] at (0,0) {$a$};
    \node[below] at (-1,-2) {$b_1$};    
    \node[below] at (-0.5,-2) {$b_2$};    
    \node[below] at (0,-2) {$b_3$};    
    \node[below] at (0.5,-2) {$b_4$};    
    \node[below] at (1,-2) {$b_5$};
    
       \end{tikzpicture}  & 
       \begin{tikzpicture}
       \filldraw (0,0) circle (2pt);
       \filldraw (0,-0.5) circle (2pt);
       \filldraw (0,-1) circle (2pt);
       \filldraw (0,-1.5) circle (2pt);
       \filldraw (0,-2) circle (2pt);
       
    \filldraw (-0.5,-0.5) circle (2pt);
    \filldraw (-0.5,-1) circle (2pt);
    \filldraw (-0.5,-1.5) circle (2pt);
    \filldraw (-0.5,-2) circle (2pt);
    
    \filldraw (-1,-0.5) circle (2pt);
    \filldraw (-1,-1) circle (2pt);
    \filldraw (-1,-1.5) circle (2pt);
    \filldraw (-1,-2) circle (2pt);
    
    \filldraw (0.5,-0.5) circle (2pt);
    \filldraw (0.5,-1) circle (2pt);
    \filldraw (0.5,-1.5) circle (2pt);
    \filldraw (0.5,-2) circle (2pt);
    
    \filldraw (1,-0.5) circle (2pt);
    \filldraw (1,-1) circle (2pt);
    \filldraw (1,-1.5) circle (2pt);
    \filldraw (1,-2) circle (2pt);
    
    \draw[thick] (0,0) -- (0.5,-0.5) -- (0.5,-2);
    \draw[thick] (0,0) -- (-0.5,-0.5) -- (-0.5,-2);
    \draw[thick] (0,0) -- (-1,-0.5) -- (-1,-2);
    \draw[thick] (0,0) -- (1,-0.5) -- (1,-2);
    \draw[thick] (0,0) -- (0,-0.5) -- (0,-2);
    \draw[thick] (-1,-2) -- (1,-2);
    
    \draw[thick] (-1,-2) -- (-0.5,-1.5);
    \draw[thick] (0,-1.5) -- (-0.5,-2);
    \draw[thick] (0,-2) -- (0.5,-1.5);
    \draw[thick] (1,-1.5) -- (0.5,-2);
    
    \node[above] at (0,0) {$a$};
    \node[below] at (-1,-2) {$b_1$};    
    \node[below] at (-0.5,-2) {$b_2$};    
    \node[below] at (0,-2) {$b_3$};    
    \node[below] at (0.5,-2) {$b_4$};    
    \node[below] at (1,-2) {$b_5$};    
    
       \end{tikzpicture} & 
       \begin{tikzpicture}
       \filldraw (0,0) circle (2pt);
       \filldraw (0,-0.5) circle (2pt);
       \filldraw (0,-1) circle (2pt);
       \filldraw (0,-1.5) circle (2pt);
       \filldraw (0,-2) circle (2pt);
       
    \filldraw (-0.5,-0.5) circle (2pt);
    \filldraw (-0.5,-1) circle (2pt);
    \filldraw (-0.5,-1.5) circle (2pt);
    \filldraw (-0.5,-2) circle (2pt);
    
    \filldraw (-1,-0.5) circle (2pt);
    \filldraw (-1,-1) circle (2pt);
    \filldraw (-1,-1.5) circle (2pt);
    \filldraw (-1,-2) circle (2pt);
    
    \filldraw (0.5,-0.5) circle (2pt);
    \filldraw (0.5,-1) circle (2pt);
    \filldraw (0.5,-1.5) circle (2pt);
    \filldraw (0.5,-2) circle (2pt);
    
    \filldraw (1,-0.5) circle (2pt);
    \filldraw (1,-1) circle (2pt);
    \filldraw (1,-1.5) circle (2pt);
    \filldraw (1,-2) circle (2pt);
    
    \draw[thick] (0,0) -- (0.5,-0.5) -- (0.5,-2);
    \draw[thick] (0,0) -- (-0.5,-0.5) -- (-0.5,-2);
    \draw[thick] (0,0) -- (-1,-0.5) -- (-1,-2);
    \draw[thick] (0,0) -- (1,-0.5) -- (1,-2);
    \draw[thick] (0,0) -- (0,-0.5) -- (0,-2);
    \draw[thick] (-1,-2) -- (1,-2);
    
    \draw[thick] (-1,-2) -- (-0.5,-1.5);
    \draw[thick] (0,-1.5) -- (-0.5,-2);
    \draw[thick] (0,-2) -- (0.5,-1.5);
    \draw[thick] (1,-1.5) -- (0.5,-2);
    
   \node[above] at (0,0) {$a$};
    \node[below] at (-1,-2) {$b_1$};    
    \node[below] at (-0.5,-2) {$b_2$};    
    \node[below] at (0,-2) {$b_3$};    
    \node[below] at (0.5,-2) {$b_4$};    
    \node[below] at (1,-2) {$b_5$};
    
    \node[right] at (1.2, -0.5) {$K_5$};
    
    \filldraw[fill=gray!20,opacity=0.2,draw=black] (0,-0.5) ellipse (1.25cm and 0.25cm);

       \end{tikzpicture} & \begin{tikzpicture}
       \filldraw (0,0) circle (2pt);
       \filldraw (0,-0.5) circle (2pt);
       \filldraw (0,-1) circle (2pt);
       \filldraw (0,-1.5) circle (2pt);
       \filldraw (0,-2) circle (2pt);
       
    \filldraw (-0.5,-0.5) circle (2pt);
    \filldraw (-0.5,-1) circle (2pt);
    \filldraw (-0.5,-1.5) circle (2pt);
    \filldraw (-0.5,-2) circle (2pt);
    
    \filldraw (-1,-0.5) circle (2pt);
    \filldraw (-1,-1) circle (2pt);
    \filldraw (-1,-1.5) circle (2pt);
    \filldraw (-1,-2) circle (2pt);
    
    \filldraw (0.5,-0.5) circle (2pt);
    \filldraw (0.5,-1) circle (2pt);
    \filldraw (0.5,-1.5) circle (2pt);
    \filldraw (0.5,-2) circle (2pt);
    
    \filldraw (1,-0.5) circle (2pt);
    \filldraw (1,-1) circle (2pt);
    \filldraw (1,-1.5) circle (2pt);
    \filldraw (1,-2) circle (2pt);
    
    \draw[thick] (0,0) -- (0.5,-0.5) -- (0.5,-2);
    \draw[thick] (0,0) -- (-0.5,-0.5) -- (-0.5,-2);
    \draw[thick] (0,0) -- (-1,-0.5) -- (-1,-2);
    \draw[thick] (0,0) -- (1,-0.5) -- (1,-2);
    \draw[thick] (0,0) -- (0,-0.5) -- (0,-2);
    
    \node[above] at (0,0) {$a$};
     \foreach \x/\y [count = \n] in {-0.75/-1.75,-0.75/-2.25,-0.75/-2.5,-0.75/-3 }
    {
          \filldraw (\x, \y) circle (1pt);
    	 \draw (\x, \y) -- (-1,-2);
          \draw (\x, \y) -- (-0.5,-2);
    }
     \foreach \x/\y [count = \n] in {-0.25/-1.75,-0.25/-2.25,-0.25/-2.5,-0.25/-3 }
    {
          \filldraw (\x, \y) circle (1pt);
    	 \draw (\x, \y) -- (0,-2);
          \draw (\x, \y) -- (-0.5,-2);
    }
    \foreach \x/\y [count = \n] in {0.25/-1.75,0.25/-2.25,0.25/-2.5,0.25/-3 }
    {
          \filldraw (\x, \y) circle (1pt);
    	 \draw (\x, \y) -- (0,-2);
          \draw (\x, \y) -- (0.5,-2);
    }
    \foreach \x/\y [count = \n] in {0.75/-1.75,0.75/-2.25,0.75/-2.5,0.75/-3 }
    {
          \filldraw (\x, \y) circle (1pt);
    	 \draw (\x, \y) -- (1,-2);
          \draw (\x, \y) -- (0.5,-2);
    }
       \end{tikzpicture}\\
         (a) & (b) & (c) & (d) 
    \end{tabular}}
    \caption{(a) $\Graph{4}$~(b) $\BGraph{4}$~(c) $\CGraph{4}$~(d) $\DGraph{4}$.}
    \label{fig:lower}
\end{figure}

{We shall provide a construction for every $k\geq 4$, this implies the statement of Theorem~\ref{thm:lower} for any $k\geq 1$.} 
First we shall prove Theorem~\ref{thm:lower}\ref{it:a}. For a fixed integer $k\geq 4$, first we describe the construction of a graph $\Graph{k}$ as follows. Consider $k+1$ paths $P_1, P_2, \ldots, P_{k+1}$ each of length $k$ and having a common endpoint $a$. For $i\in [k+1]$, let the other endpoint of $P_i$ be denoted as $b_i$. Moreover, for $i\in [k+1]$, let the neighbours of $a$ and $b_i$ in $P_i$ be denoted as $a'_i$ and $b'_i$, respectively. For $i \in [k]$, introduce an edge between $b_i$ and $b_{i+1}$. The resulting graph is denoted $\Graph{k}$ and the special vertex $a$ is the \emph{apex} of $\Graph{k}$. See Figure~\ref{fig:lower}(a). For a fixed integer $k\geq 4$, consider the graph $\Graph{k}$ and for each $i\in [k]$, introduce an edge between $b_i$ and $b'_{i+1}$. Let $\BGraph{k}$ denote the resulting graph and the special vertex $a$ is the \emph{apex} of $\BGraph{k}$. See Figure~\ref{fig:lower}(b). For a fixed integer $k\geq 4$, consider the graph $\BGraph{k}$ and for each $\{i,j\}\subseteq [k]$, introduce an edge between $a'_i$ and $a'_{j}$. Let $\CGraph{k}$ denote the resulting graph and the special vertex $a$ is the \emph{apex} of $\CGraph{k}$. See Figure~\ref{fig:lower}(c). { For a fixed integer $k\geq 4$, consider the graph $\Graph{k}$. For each $i\in [k]$, delete the edge $b_ib_{i+1}$ and introduce $k$ new vertices, each of which is adjacent to only $b_i$ and $b_{i+1}$. Call the resulting graph $W_k$. See Figure~\ref{fig:lower}(d).}

 We shall use the following result relating hyperbolicity and \emph{isometric cycles}. An induced cycle $C$ of a graph $G$ is \emph{isometric} if for any two vertices $u,v$ of $C$, the distance between $u,v$ in $C$ is the same as that in $G$.

 \begin{proposition}[Theorem 2, \cite{wu2011hyperbolicity}]\label{prp:isometric-cycle}
 Let $G$ be a graph containing an isometric cycle of order $k$ with $k \equiv c~(\bmod~4)$. Then the hyperbolicity of $G$ is at least $\lceil\frac{k}{4}\rceil - \frac{1}{2}$ if $c=1$ and $\lceil\frac{k}{4}\rceil$, otherwise.
 \end{proposition}

We now prove the following lemmas.

\begin{lemma}\label{lem:lower-1}
For $k\geq 4$, let $G$ be the graph constructed by taking two distinct copies of $\Graph{k}$ and identifying the two apex vertices. Then $G$ is a (pyramid, prism)-free graph with tree-width $2$, hyperbolicity at least $\lceil\frac{k}{2}\rceil-1$ and $\ipac{G} \geq k$.
\end{lemma}
 \begin{proof}
 Since $G$ is triangle-free, clearly $G$ is (pyramid, prism)-free. Moreover, for any induced cycle $C$ of $G$, and any vertex $w\notin C$, observe that $w$ has only one neighbour in $C$. Therefore, $G$ is also wheel-free. Observe that $G$ has an isometric cycle of length at least $2k$. Therefore, due to Proposition~\ref{prp:isometric-cycle}, $G$ has hyperbolicity at least $\lceil\frac{k}{2}\rceil-1$. Since removing the vertex $a$ from $G$ makes it acyclic, the tree-width of $G$ is two. Let $H$ and $H'$ denote the two copies of $\Graph{k}$ used to construct $G$. Let $r$ be any vertex of $G$ and, without loss of generality, assume that $r$ is a vertex of $H'$. Consider the graph $\overrightarrow{G_r}$. Now recall the construction of $H$ (which is isomorphic to $\Graph{k}$) and consider the path $Q=b_1~b_2\ldots b_k$. Observe that $Q$ is an isometric path and for any two vertices $u,v\in V(Q)$ we have $\dist{r}{u}=\dist{r}{v}$. Therefore, $\anticp{r}{Q} \geq k$. Hence, $\ipac{G} \geq k$. 
 \end{proof}


\begin{lemma}\label{lem:lower-2}
For $k\geq 4$, let $G$ be the graph constructed by taking two distinct copies of $\BGraph{k}$ and identifying the two apex vertices. Then $G$ is a (theta, prism)-free graph with tree-width $3$, hyperbolicity at least $\lceil\frac{k}{2}\rceil-1$, and $\ipac{G} \geq k$.
\end{lemma}

 \begin{proof}
 Since removing the special vertex $a$ from $G$ results in a graph with tree-width $2$, it follows that $G$ has tree-width at most $3$. Observe that $G$ has an isometric cycle of length at least $2k$. Therefore, due to Proposition~\ref{prp:isometric-cycle}, $G$ has hyperbolicity at least $\lceil\frac{k}{2}\rceil-1$. Let $H$ and $H'$ denote the two copies of $\BGraph{k}$ used to construct $G$. First we shall show that $H$ does not contain a theta or a prism. Consider the graph $H_1$ obtained by removing the apex of $H$. Observe that $H_1$ does not contain a vertex $v$ such that the vertices in $N[v]$ induce a $K_{1,3}$. Hence $H$ does not contain a theta. It also can be verified that $H_1$ does not contain a prism. Since the neighbourhood of $a$ is triangle-free, it follows that $H$ does not contain a prism. Similarly, $H'$ does not contain a theta or a prism. Now, from our construction, it follows that $G$ does not contain a theta or a prism.  Moreover, for any induced cycle $C$ of $G$, and any vertex $w\notin C$, observe that $w$ has at most two neighbours in $C$. Therefore, $G$ is wheel-free. Using arguments similar to the ones used in the proof of Lemma~\ref{lem:lower-1}, we have that $\ipac{G} \geq k$.   
 \end{proof}


\begin{lemma}\label{lem:lower-3}
For $k\geq 4$, let $G$ be the graph constructed by taking two distinct copies of $\CGraph{k}$ and identifying the two apex vertices. Then $G$ is a (theta, pyramid)-free graph with hyperbolicity at least $\lceil\frac{k}{2}\rceil-1$ and $\ipac{G} \geq k$.
\end{lemma}

 \begin{proof}
 Observe that $G$ has an isometric cycle of length at least $2k$. Therefore, due to Proposition~\ref{prp:isometric-cycle}, $G$ has hyperbolicity at least $\lceil\frac{k}{2}\rceil-1$. Let $H$ and $H'$ denote the two copies of $\BGraph{k}$ used to construct $G$. Observe that $H$ does not contain a vertex $v$ such that the vertices in $N[v]$ induce a $K_{1,3}$. Therefore, $H$ does not contain a theta or a pyramid.  Similarly, $H'$ does not contain a theta or a pyramid. Due to our construction, it follows that $G$ does not contain a theta or a pyramid. Moreover, for any induced cycle $C$ of $G$, and any vertex $w\notin C$, observe that $w$ has at most two neighbours in $C$. Therefore, $G$ is wheel-free. Using arguments similar to the ones used in the proof of Lemma~\ref{lem:lower-1}, we have that $\ipac{G} \geq k$.
 \end{proof}

 {
An isometric path cover $C$ of a graph $G$ is \emph{rooted} if there exists a vertex $v$ such that all paths in $C$ are $v$-rooted isometric paths.

\begin{lemma}\label{lem:lower-4}
For $k\geq 4$, let $G$ be the graph constructed by taking two distinct copies of $\DGraph{k}$ and identifying the two apex vertices. Then $G$ is a (prism, pyramid, wheel)-free planar graph such that any rooted isometric path cover of $G$ has cardinality at least $k^2$ but there is an isometric path cover of $G$ of cardinality $3k+1$.
\end{lemma}

\begin{proof}
      The construction ensures that $G$ is a (prism, pyramid, wheel)-free planar graph. Let $H$ and $H'$ denote the two copies of $\DGraph{k}$ used to construct $G$ and $a$ denote the apex vertex. Observe that there are $k^2$ vertices at maximum distance from the apex vertex $a$ in $H$ and a $a$-rooted isometric path can only cover one of them. Therefore, at least $k^2$ many $a$-rooted isometric paths are needed to cover the graph $H$. As $H'$ is isomorphic to $H$, it has the above properties. Since $a$ is a cut-vertex in $G$, it is easy to verify that for any vertex $v\in V(G)$, any $v$-rooted isometric path cover of $G$ requires $k^2$ many paths. On the other hand, it is easy to check that $G$ has an isometric path cover of cardinality $3k+1$. Indeed $k+1$ geodesics are sufficient to cover the vertices of the maximal isometric paths containing $a$, $2k$ geodesics are sufficient to cover the remaining vertices of $G$.
\end{proof}

}

Lemma~\ref{lem:ipac-ipco}, \ref{lem:lower-1}, \ref{lem:lower-2}, \ref{lem:lower-3}, \ref{lem:lower-4} imply Theorem~\ref{thm:lower}.






\vspace{-5pt}

\section{Conclusion}\label{sec:conclu}

In this paper, we have introduced the new graph parameter \emph{isometric path complexity}. We have shown that the isometric path complexity of a graph with $n$ vertices and $m$ edges can be computed in $O(n^2 m)$-time. It would be interesting to provide a faster algorithm to compute the isometric path complexity of a graph. Note that, no non-trivial lower bound on the achievable running time is known.

We have derived upper bounds on the isometric path complexity of three seemingly (structurally) different classes of graphs, namely hyperbolic graphs, (theta, pyramid, prism)-free graphs and outerstring graphs. An interesting direction of research is to generalise the properties of hyperbolic graphs or (theta, pyramid, prism)-free graphs to graphs with bounded isometric path complexity.

Note that, in our proofs we essentially show that, for any graph $G$ that belongs to one of the above graph classes, any vertex $v$ of $G$, and any isometric path $P$ of $G$, the path $P$ can be covered by a small number of $v$-rooted isometric paths. This implies our ``choice of the root'' is arbitrary. This motivates the following definition. The \emph{strong isometric path complexity} of a graph $G$
is the minimum integer $k$ such that for each vertex $v\in V(G)$ we have that the vertices of any isometric path $P$ of $G$ can be covered by $k$ many $v$-rooted isometric paths. Our proofs imply that the strong isometric path complexity of graphs from all the graph classes addressed in this paper are bounded. 
We also wonder whether one can find other interesting graph classes with small (strong) isometric path complexity. 

Our results imply a constant-factor approximation algorithm for \IPC on hyperbolic graphs, (theta, pyramid, prism)-free graphs and outerstring graphs. However, the existence of a constant-factor approximation algorithm for \IPC on general graphs is not known (an $O(\log n)$-factor approximation algorithm is designed in~\cite{TG21}).

\bibliographystyle{plain}
\bibliography{references}

\begin{thebibliography}{10}

\bibitem{aboulker2015wheel}
P.~Aboulker, M.~Chudnovsky, P.~Seymour, and N.~Trotignon.
\newblock Wheel-free planar graphs.
\newblock {\em European Journal of Combinatorics}, 49:57--67, 2015.

\bibitem{cop-decs}
I.~Abraham, C.~Gavoille, A.~Gupta, O.~Neiman, and K.~Talwar.
\newblock Cops, robbers, and threatening skeletons: Padded decomposition for minor-free graphs.
\newblock {\em SIAM Journal on Computing}, 48(3):1120--1145, 2019.

\bibitem{AF84}
M.~Aigner and M.~Fromme.
\newblock A game of cops and robbers.
\newblock {\em Discrete Applied Mathematics}, 8(1):1--12, 1984.

\bibitem{BELMONTE201354}
R.~Belmonte and M.~Vatshelle.
\newblock Graph classes with structured neighborhoods and algorithmic applications.
\newblock {\em Theoretical Computer Science}, 511:54--65, 2013.

\bibitem{biedl2018size}
T.~Biedl, A.~Biniaz, and M.~Derka.
\newblock On the size of outer-string representations.
\newblock In {\em Proceedings of the 16th Scandinavian Symposium and Workshops on Algorithm Theory (SWAT 2018)}, 2018.

\bibitem{twinwidth}
\'{E}. Bonnet, E.~J. Kim, S.~Thomass\'{e}, and R.~Watrigant.
\newblock Twin-width {I}: Tractable fo model checking.
\newblock {\em Journal of the ACM}, 69(1), 2021.

\bibitem{bose2022computing}
P.~Bose, P.~Carmi, J.~M. Keil, A.~Maheshwari, S.~Mehrabi, D.~Mondal, and M.~Smid.
\newblock Computing maximum independent set on outerstring graphs and their relatives.
\newblock {\em Computational Geometry}, 103:101852, 2022.

\bibitem{cardinal2017intersection}
J.~Cardinal, S.~Felsner, T.~Miltzow, C.~Tompkins, and Birgit Vogtenhuber.
\newblock Intersection graphs of rays and grounded segments.
\newblock In {\em Proceedings of the International Workshop on Graph-Theoretic Concepts in Computer Science (WG 2017)}, pages 153--166. Springer, 2017.

\bibitem{shortversion}
D.~Chakraborty, J.~Chalopin, F.~Foucaud, and Y.~Vax{\`{e}}s.
\newblock Isometric path complexity of graphs.
\newblock In J{\'{e}}r{\^{o}}me Leroux, Sylvain Lombardy, and David Peleg, editors, {\em 48th International Symposium on Mathematical Foundations of Computer Science ({MFCS} 2023)}, volume 272 of {\em LIPIcs}, pages 32:1--32:14. Schloss Dagstuhl - Leibniz-Zentrum f{\"{u}}r Informatik, 2023.

\bibitem{ChakrabortyD0FG22}
D.~Chakraborty, A.~Dailly, S.~Das, F.~Foucaud, H.~Gahlawat, and S.~K. Ghosh.
\newblock Complexity and algorithms for {ISOMETRIC} {PATH} {COVER} on chordal graphs and beyond.
\newblock In {\em Proceedings of the 33rd International Symposium on Algorithms and Computation ({ISAAC} 2022)}, volume 248 of {\em LIPIcs}, pages 12:1--12:17, 2022.

\bibitem{chepoi2008diameters}
V.~Chepoi, F.~Dragan, B.~Estellon, M.~Habib, and Y.~Vax{\`e}s.
\newblock Diameters, centers, and approximating trees of $\delta$-hyperbolic geodesic spaces and graphs.
\newblock In {\em Proceedings of the 24th {ACM} Symposium on Computational Geometry (SCG 2008)}, pages 59--68. {ACM}, 2008.

\bibitem{chepoi2019fast}
V.~Chepoi, F.~Dragan, M.~Habib, Y.~Vax{\`e}s, and H.~Alrasheed.
\newblock Fast approximation of eccentricities and distances in hyperbolic graphs.
\newblock {\em Journal of Graph Algorithms and Applications}, 23(2):393--433, 2019.

\bibitem{chepoi2017core}
V.~Chepoi, F.~F. Dragan, and Y.~Vaxes.
\newblock Core congestion is inherent in hyperbolic networks.
\newblock In {\em Proceedings of the Twenty-Eighth Annual ACM-SIAM Symposium on Discrete Algorithms (SODA 2017)}, pages 2264--2279. SIAM, 2017.

\bibitem{chudnovsky2006strong}
M.~Chudnovsky, N.~Robertson, P.~Seymour, and R.~Thomas.
\newblock The strong perfect graph theorem.
\newblock {\em Annals of mathematics}, pages 51--229, 2006.

\bibitem{conforti1997universally}
M.~Conforti, G.~Cornu{\'e}jols, A.~Kapoor, and K.~Vu{\v{s}}kovi{\'c}.
\newblock Universally signable graphs.
\newblock {\em Combinatorica}, 17(1):67--77, 1997.

\bibitem{corneil2013ldfs}
D.~G. Corneil, B.~Dalton, and M.~Habib.
\newblock {LDFS}-based certifying algorithm for the minimum path cover problem on cocomparability graphs.
\newblock {\em SIAM Journal on Computing}, 42(3):792--807, 2013.

\bibitem{corneil1997asteroidal}
D.~G. Corneil, S.~Olariu, and L.~Stewart.
\newblock Asteroidal triple-free graphs.
\newblock {\em SIAM Journal on Discrete Mathematics}, 10(3):399--430, 1997.

\bibitem{coudert2021enumeration}
D.~Coudert, A.~Nusser, and L.~Viennot.
\newblock Enumeration of far-apart pairs by decreasing distance for faster hyperbolicity computation.
\newblock {\em {ACM} Journal of Experimental Algorithmics}, 27:1.15:1--1.15:29, 2022.

\bibitem{das2018effect}
B.~Das~Gupta, M.~Karpinski, N.~Mobasheri, and F.~Yahyanejad.
\newblock Effect of {G}romov-hyperbolicity parameter on cuts and expansions in graphs and some algorithmic implications.
\newblock {\em Algorithmica}, 80(2):772--800, 2018.

\bibitem{davies2021circle}
J.~Davies and R.~McCarty.
\newblock Circle graphs are quadratically $\chi$-bounded.
\newblock {\em Bulletin of the London Mathematical Society}, 53(3):673--679, 2021.

\bibitem{D50}
R.~P. Dilworth.
\newblock A decomposition theorem for partially ordered sets.
\newblock {\em Annals of Mathematics}, 51(1):161--166, 1950.

\bibitem{diot2020theta}
\'E. Diot, M.~Radovanovi{\'c}, N.~Trotignon, and K.~Vu{\v{s}}kovi{\'c}.
\newblock The (theta, wheel)-free graphs {Part I}: only-prism and only-pyramid graphs.
\newblock {\em Journal of Combinatorial Theory, Series B}, 143:123--147, 2020.

\bibitem{francis2014forbidden}
M.~Francis, P.~Hell, and J.~Stacho.
\newblock Forbidden structure characterization of circular-arc graphs and a certifying recognition algorithm.
\newblock In {\em Proceedings of the Twenty-Sixth Annual ACM-SIAM Symposium on Discrete Algorithms (SODA 2014)}, pages 1708--1727. SIAM, 2014.

\bibitem{gromov1987}
M.~Gromov.
\newblock Hyperbolic groups.
\newblock In {\em Essays in group theory}, pages 75--263. Springer, 1987.

\bibitem{JAFFKE2020153}
Lars Jaffke, O~joung Kwon, and Jan~Arne Telle.
\newblock Mim-width {I}. induced path problems.
\newblock {\em Discrete Applied Mathematics}, 278:153--168, 2020.
\newblock Eighth Workshop on Graph Classes, Optimization, and Width Parameters.

\bibitem{keil2017algorithm}
J.~M. Keil, J.S.B Mitchell, D.~Pradhan, and M.~Vatshelle.
\newblock An algorithm for the maximum weight independent set problem on outerstring graphs.
\newblock {\em Computational Geometry}, 60:19--25, 2017.

\bibitem{kosowski2015k}
A.~Kosowski, B.~Li, N.~Nisse, and K.~Suchan.
\newblock {K}-chordal graphs: From cops and robber to compact routing via treewidth.
\newblock {\em Algorithmica}, 72(3):758--777, 2015.

\bibitem{kratochvil1991string}
J.~Kratochv{\'\i}l.
\newblock String graphs. {I}. {T}he number of critical nonstring graphs is infinite.
\newblock {\em Journal of Combinatorial Theory, Series B}, 52(1):53--66, 1991.

\bibitem{rok2019outerstring}
A.~Rok and B.~Walczak.
\newblock Outerstring graphs are $\chi$-bounded.
\newblock {\em SIAM Journal on Discrete Mathematics}, 33(4):2181--2199, 2019.

\bibitem{shavitt2004curvature}
Y.~Shavitt and T.~Tankel.
\newblock On the curvature of the internet and its usage for overlay construction and distance estimation.
\newblock In {\em Proceedings of the 23rd Annual Joint Conference of the IEEE Computer and Communications Societies (IEEE INFOCOM 2004)}. IEEE, 2004.

\bibitem{TG21}
M.~Thiessen and T.~Gaertner.
\newblock Active learning of convex halfspaces on graphs.
\newblock In {\em Proceedings of the 35th Conference on Neural Information Processing Systems (NeurIPS 2021)}, volume~34, pages 23413--23425. Curran Associates, Inc., 2021.

\bibitem{trotignonprivate}
N.~Trotignon.
\newblock Private communication, 2022.

\bibitem{vuskovic2013world}
K.~Vu\v{s}kovi\'c.
\newblock The world of hereditary graph classes viewed through {T}ruemper configurations.
\newblock {\em Surveys in Combinatorics 2013}, 409:265, 2013.

\bibitem{walter2002interactive}
J.~A. Walter and H.~Ritter.
\newblock On interactive visualization of high-dimensional data using the hyperbolic plane.
\newblock In {\em Proceedings of the eighth ACM SIGKDD international conference on Knowledge discovery and data mining (SIGKDD 2002)}, pages 123--132, 2002.

\bibitem{watkins1967cycles}
M.~E. Watkins and D.~M. Mesner.
\newblock Cycles and connectivity in graphs.
\newblock {\em Canadian Journal of Mathematics}, 19:1319--1328, 1967.

\bibitem{wu2011hyperbolicity}
Y.~Wu and C.~Zhang.
\newblock Hyperbolicity and chordality of a graph.
\newblock {\em The Electronic Journal of Combinatorics}, 18(1):P43, 2011.

\end{thebibliography}

\end{document}